\documentclass[preprint,sort&compress,12pt]{elsarticle}
\usepackage[pagewise]{lineno}
\usepackage{bbm}
\usepackage{amstext}
\usepackage{amsmath}
\usepackage{amssymb}
\usepackage{mathrsfs}
\usepackage{stmaryrd}
\usepackage{bm}
\usepackage{pstricks}
\usepackage{pst-coil}
\usepackage{pst-3d}
\usepackage{amsfonts}
\usepackage{amsthm}
\usepackage{CJK}
\allowdisplaybreaks

\makeatletter
\def\ps@pprintTitle{%
   \let\@oddhead\@empty
   \let\@evenhead\@empty
   \def\@oddfoot{\reset@font\hfil\thepage\hfil}%
   \let\@evenfoot\@oddfoot
}
\makeatother

\newcommand{\mf}{\mathbf}
\newcommand{\mm}{\mathrm}

\newcommand{\be}{\begin{equation}}
\newcommand{\bea}{\begin{equation}\begin{aligned}}
\newcommand{\beas}{\begin{equation*}\begin{aligned}}
\newcommand{\eeas}{\end{aligned}\end{equation*}}
\newcommand{\eea}{\end{aligned}\end{equation}}
\newcommand{\ee}{\end{equation}}

\begin{document}
\begin{CJK*}{GBK}{song}
	\begin{frontmatter}
		\title{On Global-in-time Solutions of Incompressible \\
MHD Equations with Small Alfv\'en Numbers} 
		\author[mmL1]{Fei Jiang}
		\ead{jiangfei0591@163.com}
\author[mmL12]{Xiao Ren}
\ead{xren@fudan.edu.cn}
		\author[Yuan]{Yi Zhou\corref{cor1}}
		\ead{pde_in_physics@163.com} 
\cortext[cor1]{Corresponding author.}
		\address[mmL1]{School of Mathematics and Statistics, Fuzhou University, Fuzhou, 350108, China.}
\address[mmL12]{Center for Applied Mathematics, Fudan University, Shanghai 200433, China.}
\address[Yuan]{School of Mathematical Sciences, Fudan University, Shanghai 200433, China.}		
\begin{abstract}
In 1965 Kraichnan pointed out that
 a {sufficiently} strong background magnetic field, i.e. the case of \emph{small Alfv\'en number},  will reduce the nonlinear interaction and inhibit the formation of strong gradients in the magnetohydrodynamic (abbr. MHD)  system with $\mu=\nu\geqslant 0$, where ${\mu}$ and $\nu $ are the  coefficients of kinematic viscosity and resistivity resp.. This means that the MHD system with ${\mu}=\nu\geqslant 0$ admits
global-in-time large perturbation solutions with small Alfv\'en numbers. The existence of such large perturbation solutions was first mathematically verified in H\"older spaces by Bardos--Sulem--Sulem for the case ${\mu}=\nu= 0$ in 1988, and in Sobolev spaces by Cai--Cui--Jiang--Liu for the case ${\mu}=\nu> 0$ recently. In this paper, we further found a similar result for the general case ``${\mu}>0$ and $\nu>0$", and provide a rigorous proof by developing a new approach, which includes a key bilinear estimate for dealing with the nonlinear interaction terms. Moreover both additional results for the vanishing behavior of the nonlinear interaction and the small Alfv\'en number limit of solutions  are also established.
\end{abstract}
\end{frontmatter}
\newtheorem{thm}{Theorem}[section]
\newtheorem{lem}{Lemma}[section]
\newtheorem{pro}{Proposition}[section]
\newtheorem{cor}{Corollary}[section]
\newproof{pf}{Proof}
\newdefinition{rem}{Remark}[section]
\newtheorem{definition}{Definition}[section]
\newtheorem{concl}{Assertion}[section]
	
\section{Introduction}\label{sect1}
\numberwithin{equation}{section}
The basic motion equations describing the incompressible magnetohydrodynamic (abbr. MHD) flow in $\mathbb{R}^n$ with $n=2$, $3$ read as follows \cite{MR346289}
\begin{equation}\label{1.1}
\begin{cases}
\rho v_t+  \rho v\cdot\nabla v- \tilde{\mu}\Delta v+\nabla P =\lambda  M\cdot\nabla M/(4\pi) ,\\[1mm]
M_t+v\cdot\nabla M-\nu\Delta M=M\cdot\nabla v,\\[1mm]
\mathrm{div}v=\mathrm{div}{M}=0.
\end{cases}
\end{equation}
The unknown functions ${v}:= {v}(x,t)$, $M:= {M}(x,t)$ and $P:= P(x,t)$ denote the velocity, the magnetic field and the sum of kinetic  pressure and magnetic pressure, resp.. The nonnegative constants $\rho> 0$ and $\lambda$ are the density and the magnetic permeability, resp. In addition, $\tilde{\mu}\geqslant 0 $ and $\nu \geqslant 0$ represent the  coefficients of viscosity and resistivity, resp.. 

\subsection{Motivation}
Now we consider a rest state $(0,\bar{M})$ for the system \eqref{1.1}, and then define
\begin{align}
H:=\sqrt{\frac{\lambda}{4\pi\rho}}\left(M-\bar{M}\right),\ p:=\frac{P}{\rho},\  \bar{M}:= \varpi\mf{e}^{1},\ {\mu}:=\frac{\tilde{\mu}}{\rho}\mbox{ and }\varepsilon:=\sqrt{\frac{4\pi\rho}{\lambda \varpi^2}},
\label{2511040930}
\end{align}
 where the constant vector $\bar{M}\in\mathbb{R}^n$ is a background (or impressed) magnetic field,  $\mathbf{e}^1\in \mathbb{R}^n$ the unit vector with the first component being $1$,   and $\mu$, resp. $\varepsilon$ are called the kinematic viscosity coefficient resp. \emph{Alfv\'en number} in this paper. Thus we have the following perturbation system for $(v,H)$:
\begin{equation}\label{1sa.1}
\begin{cases}
v_t+  v\cdot\nabla v-  {\mu}\Delta v+\nabla p =(\varepsilon^{-1} \mathbf{e}^1+ H )\cdot \nabla H, \\[1mm]
H_t+v\cdot\nabla H-\nu\Delta H=(\varepsilon^{-1}  \mf{e}^{1} +H)\cdot\nabla v,\\[1mm]
\mathrm{div}v=\mathrm{div}H=0.
\end{cases}
\end{equation}
 The existence of the global(-in-time) strong solutions of \eqref{1sa.1} with small initial data has been widely investigated, see \cite{MR3780143,MR3738384,MR3678491,MR4072209} for the case $\mu=\nu=0$, \cite{MR3666563,MR3851055,MR3377532,MR3210038,MR3296601} for the case ``$\mu>0$ and $\nu=0$", and \cite{MR4372923,MR4340933,MR3843627} for the cases ``$\mu=0$ and $\nu>0$". We mention that the two-dimensional (abbr. 2D) system \eqref{1sa.1} with general initial data admits unique global strong solutions for the case ``$\mu>0$ and $\nu>0$" \cite{MR716200}. In addition, interested readers refer to \cite{MR716200,MR3590662,MR3217057,MR3682684,MR3415680} and the references cited therein for the local(-in-time) solutions of \eqref{1sa.1} with large initial data.
 
Kraichnan pointed out that, in the MHD system \eqref{1sa.1} with $\mu=\nu\geqslant 0$, a {sufficiently} strong  $\bar{M} $ (i.e. sufficiently small $\varepsilon$ by recalling the definitions of $\bar{M}$ and $\varepsilon$ in \eqref{2511040930}) will reduce the nonlinear interaction and inhibit the formation of strong gradients \cite{MR1927281}.
This effect was also observed in direct numerical simulations of the ideal MHD system (i.e. the system \eqref{1sa.1} with $\mu=\nu=0$) with periodic boundary conditions \cite{MR7201451}. In 1998, Bardos--Sulem--Sulem first mathematically proved that the ideal MHD system in the H\"older space is globally well-posed for large initial perturbations around the equilibrium state $(0,\bar{M})$ by exploiting the hyperbolicity structure, when 
 $\varepsilon$ is sufficiently small (i.e. the field strength $|\bar{M}|$ is sufficiently large) \cite{MR920153}. Later Zhang rewrote the 2D MHD system \eqref{1sa.1} as wave equations with dissipative terms for the case ``$\mu>0$ and $\nu=0$", and then established the existence result of large perturbation solutions with small $\varepsilon$ \cite{MR3448784} by using the oscillatory integral theory, also see Jiang--Jiang for the 3D case by exploiting magnetic inhibition theory in Lagrange coordinates \cite{MR4641656}.
Recently, Cai--Cui--Jiang--Liu further obtained a similar result for the system \eqref{1sa.1} with $\mu=\nu>0$ by using Alinhac's ghost weight technique \cite{cai2025small}. Thus the result of Cai et.al., together with the one of Bardos et.al. completes the mathematical proof of  Kraichnan's assertion for $\mu=\nu\geqslant 0$.

However, motivated by the existence results of large perturbation solutions with small $\varepsilon$ in \cite{MR3448784,MR4641656} for the case $\mu\neq \nu =0$, ones naturally conjecture that Kraichnan's assertion may be extended to the general cases $\mu$, $\nu\geqslant 0$, that is (roughly speaking): 
\begin{concl}\label{2510071030} Let $\mu$, $\nu\geqslant 0$ be given. If
 \begin{align}
 \label{2510071047}
\|(v^0,H^0)\|_X  \ll \varepsilon^{-s} 
 \end{align}
 for some $s>0$,
 then the Cauchy problem of the perturbation system \eqref{1sa.1} admits a global(-in-time) strong solution with initial data $(v^0,H^0)$.
 Here 
  $\|(v^0,H^0)\|_X$ represents some norm of $(v^0,H^0)$ and  $A\ll B$ means that $A$ is much smaller than $B$.
\end{concl} 

For reader's convenience, we \emph{roughly} summarize the previous works for Assertion \ref{2510071030} as follows. 
\begin{center}
\renewcommand{\arraystretch}{1.2} 
\begin{tabular}{|c|c|c|} 
\hline Values of $\mu$ and $\nu$   & Whether Assertion \ref{2510071030} holds?  &  References \\ 
   \hline
 $\mu=\nu=0 $ & Yes & \cite{MR920153,cai2025small}\\
   \hline
$\mu=\nu>0 $  & Yes & \cite{cai2025small}\\  \hline
 $\mu>0$, $\nu=0 $  & Yes & \cite{MR4641656,MR3448784}\\ \hline
\end{tabular} 
\end{center}
In this paper, we will further prove that Assertion \ref{2510071030} holds for the case ``$\mu>0$ and $\nu>0$" by developing a new proof framework, which is quite different from the ones in \cite{MR920153,cai2025small,MR4641656,MR3448784}. Moreover the global solutions also enjoy the vanishing behavior of the nonlinear interaction. It should be noted that our approach strongly depends on the viscosity and the resistivity, and thus can not be applied to the remaining case ``$\mu=0$ and $\nu>0$", which should be further deeply investigated. 
 Finally, it is interesting that the stability condition \eqref{2510071047} in Assertion \ref{2510071030} is very similar to the one of Couette flow in the MHD system \cite{MR4729862,MR4115008,MR4747502}.
 
\subsection{Reformulation}	
To conveniently investigate the large perturbation solutions with small $\varepsilon$, we shall first rewrite the system \eqref{1sa.1}.  
  Noting that the initial condition \eqref{2510071047} with $s=1$ is equivalent to 
$$ \varepsilon \|(v^0,H^0)\|_X  \ll 1,$$
thus we use the new variables
\begin{align}
\label{2511042225}
 u :=\varepsilon v(x,\varepsilon  t),\  h:=\varepsilon H(x,\varepsilon  t)\mbox{ and }q:= \varepsilon^2 p(x,\varepsilon  t) 
\end{align}
to rewrite the system \eqref{1sa.1} for $(u,h)$, and then obtain the following Cauchy problem 
\begin{equation}\label{1.4}
\begin{cases}
u_t+  u\cdot\nabla u- \varepsilon {\mu}\Delta u+\nabla q =( \mathbf{e}^1+ h )\cdot \nabla h, \\[1mm]
h_t+u\cdot\nabla h-\varepsilon\nu\Delta h=(\mf{e}^{1} +h)\cdot\nabla u,\\[1mm]
\mathrm{div}u=\mathrm{div}h=0,\\[1mm]
u|_{t=0}=\varepsilon v^0,\ h|_{t=0}=\varepsilon H^0.
\end{cases}
\end{equation}

In next section, we will present the existence result of global-in-time solutions of the above Cauchy problem enjoying a uniform energy estimate with  respect to $\varepsilon$ in Theorem \ref{thm1}, and then the error estimate with respect to $\varepsilon$ between the (nonlinear) solution of the Cauchy problem  \eqref{1.4} and the one of the corresponding linear (pressureless) Cauchy problem in Theorem \ref{thma2}. Thanks to Theorems \ref{thm1} and \ref{thma2},  we easily further record the existence result and the vanishing phenomenon of the nonlinear interaction for the solutions of the original perturbation system \eqref{1sa.1} in Theorem \ref{tsafhm1}.  In particular, the solutions have a uniform estimate with respect to $\varepsilon$ (see \eqref{2602142045} in Theorem \ref{tsafhm1}), this fact naturally motives ones to further investigate the limiting behavior of solutions as $\varepsilon\to 0$, i.e. the well-known small Alfv\'en number limit problem, the relevant progresses for which will be  complementarily introduced in Section \ref{sec2}. In this paper we will exploit the oscillatory integral structures of the linear fundamental solution with respect to $\varepsilon$ and the vanishing behavior of the nonlinear interaction to conclude that the small Alfv\'en number limit of solutions in Theorem \ref{tsafhm1} is zero, see Theorem \ref{tsafasdhm1} for details.
 
\section{Main results}
Before stating the main results, we shall introduce some notations used frequently throughout this paper.

(1) Basic notations: $\mathbb{R}_{0}^{+}:=[0,\infty)$. Let $x\in \mathbb{R}^n$, then $x_{\mm{v}}:=x_2$ for $n=2$, and $x_{\mm{v}}:=(x_2,x_3)^\top$ for $n=3$. $\partial_i:=\partial_{x_i}$, $\int:=\int_{\mathbb{R}^{n}}$, $\alpha:=\left(\alpha_{1},  \cdots,\alpha_{n}\right)$ is the multi-index with respect to the variable $x$, i.e. $\partial^\alpha=\partial_{1}^{\alpha_1}\cdots\partial_{n}^{\alpha_n}$. If  the two multi-indexes $\alpha$ and $\beta$ satisfy $\beta_i\leqslant \alpha_i$ for any $1\leqslant i\leqslant n$, we denote such relation by $\beta\leqslant \alpha$. We simplify the inequality
 $a\leqslant c b$ by $a\lesssim b$, where $c$ represents a generic positive constant, which at most depends on $\mu$, $\nu$, $n$ and the regularity index $m$ of initial data appearing in Theorem \ref{thm1}, and may vary from line to line. Similarly $a\leqslant c(\chi_1,\cdot,\chi_\ell) b$ by $a\lesssim_{\chi_1,\cdot,\chi_\ell} b$, where the positive constant $ c(\chi_1,\cdot,\chi_\ell)$ further depends on the parameters $\chi_1$, $\cdot$, $\chi_\ell$, and also may vary from line to line. 
	
(2) Simplified norms and function spaces:
\begin{align*}
&L^2:=L^2(\mathbb{R}^n),\ H^{j}:=W^{j, 2}\left(\mathbb{R}^{n}\right), \ H^j_{\sigma}:=\{w\in H^j~|~\mm{div}\,w=0\},\\
&\|\cdot\|_j:=\|\cdot\|_{H^j},\  \|\phi(x)\|_{L^r_{{x}_{1}}}:=\|\phi(x_1,\cdot)\|_{L^r(\mathbb{R})}\mbox{ for given }x_{\mm{v}}  ,\\ 
&\|f(x)\|^2_{H^j_{ {x}_{\mm{v}}}}:=\sum_{|\alpha|=0}^j\|(\partial^\alpha f)(\cdot,x_{\mm{v}})\|^2_{L^2(\mathbb{R}^{n-1})}\mbox{ for given }x_1,\\ &\|f(x_1,t)\|_{L^r_{a,x_1}}:=\|f(x_1, t)\|_{L^r((0,a)_t\times \mathbb{R})},\ \|f(x,t)\|_{L^r_aL^\tau}:=
\|f(x,t)\|_{L^r((0,a); L^\tau(\mathbb{R}^n))}, 
\end{align*}
where $j\geqslant 0$ is integer, $0<a\leqslant \infty$ and $1 \leqslant r$, $\tau\leqslant \infty$. 
 
\subsection{Results for the Cauchy problem \eqref{1.4}}
We introduce the first result concerning the existence of unique global-in-time solutions to the Cauchy problem \eqref{1.4}  with  a uniform estimate with respect to  the Alfv\'en number $\varepsilon $.
\begin{thm}\label{thm1}
Assume $n=2$, $3$, $\mu>0$, $\nu> 0$ and  
$(v^0,H^0)\in H^m_\sigma$  with the integer $m\geqslant 3$. There {exists a small constant}  $\delta\in(0,1)$, which depends only on $\mu$, $\nu$, $n$ and $m$ such that, for any $\varepsilon$ satisfying the condition of small  Alfv\'en number 
\begin{equation}\label{1.6}
\varepsilon  \max\{ \|(v^0,H^0) \|_m^2, \|(v^0,H^0) \|_m^3 \}\leqslant  \delta,
\end{equation}
the Cauchy problem \eqref{1.4} defined on $\mathbb{R}^n\times \mathbb{R}_0^+$ admits a unique global-in-time classical solution $(u,h)\in C^0(\mathbb{R}_0^+,H^m\times H^m)$. Moreover the solution satisfies
\begin{align}
\label{2504160945}
\|(u,h)(t)\|_m + \sqrt{\varepsilon}\|\nabla (u,h)\|_{L^2_tH^{m}} \lesssim \varepsilon \|(v^0,H^0)\|_m 
\end{align} for any $t>0$.
\end{thm}

It's well-known that the Cauchy problem \eqref{1.4} admits a unique local(-in-time) solution for the generic initial data $(v^0,H^0)$. Therefore the key proof of Theorem \ref{thm1} is to  derive the stability estimate \eqref{2504160945} under some \emph{a priori} assumptions. Next we briefly sketch how to obtain \eqref{2504160945}.

 It is easy to see that the solution $(u,h)$ of \eqref{1.4} enjoys the following basic energy equality:
\begin{align}
\|u(t)\|_0^2 + \|h(t)\|_0^2 + 2\varepsilon\int^t_0( \mu  \|\nabla u\|_0^2 + \nu  \|\nabla h\|_0^2)\mm{d}\tau= 
\varepsilon^2 ( \|v^0\|_0^2 + \|H^0\|_0^2).
\label{2511041244}
\end{align}
Motivated by the basic energy equality above, we naturally expect that $(u,h)$ defined on $[0,T]$ may admit the following stability estimate
\begin{align} 
\label{2511041843}
E(T)+D(T):= \sup_{0 \leqslant t \leqslant T}\{\|(u,h)(t)\|_m^2\} + \varepsilon \|\nabla (u,h)\|_{L^2_TH^{m}}^2\lesssim \varepsilon^2 \|(v^0,H^0)\|_m^2 
\end{align}
under the \emph{a priori} assumption
\begin{align}
 \sup_{0 \leqslant t \leqslant T}\{\varepsilon^{-1} \|(u,h)(t)\|_{m}\} \leqslant K
\label{2510081829}
\end{align} with small Alfv\'en number $\varepsilon$.

Following a standard energy method from \eqref{2511041244} to \eqref{2511041843}, ones inevitably face some integrals involving nonlinear terms, which are called the nonlinear integrals for the simplicity, such as
$$ 
I:= \int_0^T\int  \partial^\alpha h\cdot\nabla h\cdot\partial^\alpha u \mm{d}x\mm{d}t\mbox{ with }|\alpha|=m.
 $$ 
 It is easy to see that 
\begin{align} 
|I|\lesssim \sup_{0 \leqslant t \leqslant T}\{\|h(t)\|_{m-1}\} \|\nabla (u,h)\|_{L^2_TH^{m-1}}^2 \lesssim 
\varepsilon K \|\nabla (u,h)\|_{L^2_TH^{m-1}}^2. 
\label{2511061718}
 \end{align}
Obviously ones can not use the above standard estimate to establish \eqref{2511041843},  unless 
\begin{align}
 \varepsilon\mbox{ in \eqref{2511061718} can be replaced by }\varepsilon^\vartheta\mbox{ with }\vartheta>1.
\label{2511061716}
\end{align}

To realize the idea in \eqref{2511061716}, we shall introduce \begin{align*}
\Lambda_{\pm}: =u \pm h
\end{align*}
to rewrite all nonlinear integrals, and then find that  (see Lemma \ref{2511041741} for the details)
\begin{align}
&\mbox{the sum of all nonlinear integrals  can be controlled by}\nonumber 
\\
&\left\|\| \nabla \Lambda_+\|_{H^{m-2}_{{x}_{\mm{v}}}} \| \nabla \Lambda_-\|_{H^{m-1}_{{x}_{\mm{v}}}} +\| \nabla \Lambda_-\|_{H^{m-2}_{{x}_{\mm{v}}}}\| \nabla \Lambda_+\|_{H^{m-1}_{{x}_{\mm{v}}}} \right\|_{L^2_{T,x_1}} \|  \nabla (u, h)\|_{L^2_TH^{m}}.
\label{202511041943}
\end{align}Exploiting the div-curl lemma in \cite{lai2025global,MR01542012}, we can deduce from \eqref{1.4} that (see Lemma \ref{2511041742})
\begin{align}
\label{25100712411}
& \left\|\|\nabla \Lambda_+\|_{H^{m-2}_{{x}_{\mm{v}}}}\| \nabla \Lambda_-
\|_{H^{m-1}_{{x}_{\mm{v}}}}+\| \nabla \Lambda_-
\|_{H^{m-2}_{{x}_{\mm{v}}}}\|\nabla \Lambda_+\|_{H^{m-1}_{{x}_{\mm{v}}}}\right\|_{L^2_{T,x_1}} \nonumber 
\\
&\lesssim  \sup_{0 \leqslant t \leqslant T}\{\|  (u,h)(t) \|_{m}^2\}\nonumber \\
&\quad +\sup_{0 \leqslant t \leqslant T}\{\|  (u,h)(t) \|_{m}\}
\left(\sup_{0 \leqslant t \leqslant T}\{\|  (u,h)(t) \|_{m}\}
 +{\varepsilon}\right)^{1/2}\| \nabla (u,h) \|_{L^2_TH^{m}},
\end{align}
which is called by the bilinear estimate.

Thanks to \eqref{202511041943} and \eqref{25100712411}, ones easily see that, under the assumption \eqref{2510081829},
\begin{align}
&\mbox{the sum of all nonlinear integrals can be controlled by}\nonumber\\
&\sup_{0 \leqslant t \leqslant T}\{\|  (u,h)(t) \|_{m}\}
\Bigg(\sup_{0 \leqslant t \leqslant T}\{\|  (u,h)(t) \|_{m}\}\nonumber \\
&+ \left(\sup_{0 \leqslant t \leqslant T}\{\|  (u,h)(t) \|_{m}\}
+{\varepsilon} \right)^{1/2}\| \nabla (u,h) \|_{L^2_TH^{m}} \Bigg) \| \nabla (u,h) \|_{L^2_TH^{m}} \nonumber \\
  &\lesssim\varepsilon K\left( \sup_{0 \leqslant t \leqslant T}\{\|  (u,h)(t) \|_{m}\}  +\varepsilon^{1/2} (K^{1/2}+1)\| \nabla (u,h) \|_{L^2_TH^{m}}\right) \|\nabla (u,h) \|_{L^2_TH^{m}}.
  \label{2511042004}
\end{align}
In particular, under the  smallness condition 
\begin{align}
\label{2511042007}
\varepsilon^{1/2}\max\{K,K^{3/2}\}\leqslant \delta\ll 1,
\end{align} 
ones easily observe from \eqref{2511042004} that all nonlinear integrals  can be absorbed by $E(T)$ and $D(T)$ defined in \eqref{2511041843}, and thus the desired stability estimate follows
\begin{align} E(T)+D(T)\leqslant(c \varepsilon  \|(v^0,H^0)\|_m)^2 =: K^2/4 
\label{11061426}
\end{align} 
 under the \emph{a priori} assumptions \eqref{2510081829} and \eqref{2511042007}. Based on this fact, ones can easily extend the unique local solutions to the global ones, and thus obtain Theorem \ref{thm1}.

Next we state the second result concerning the error estimate between the (nonlinear) solution in Theorem \ref{thm1} and the one
 of the corresponding (pressureless) linear problem. 
\begin{thm}\label{thma2}
Let $(u,h)$ be the solution in Theorem \ref{thm1} and 
$(u^\mm{L},h^\mm{L})$ the (unique) solution to the linear (pressureless) problem defined on $\mathbb{R}^n\times \mathbb{R}_0^+$:
 \begin{align}\label{1.11}
 &\begin{cases}
u_t^\mm{L} - \varepsilon {\mu}\Delta u^\mm{L}=\partial_1 h^\mm{L}, \\[1mm]
h_t^\mm{L}-\varepsilon\nu\Delta h^\mm{L}=\partial_1u^\mm{L},\\[1mm] 
u^\mm{L}|_{t=0}=\varepsilon v^0,\ h^\mm{L}|_{t=0}=\varepsilon H^0,
 \end{cases} 
 \end{align}  
 where $v^0$, $ H^0$ satisfy the assumptions in Theorem \ref{thm1}, and the linear solution $(u^\mm{L},h^\mm{L})$ enjoys the regularity as well as $(u,h)$. Then it holds that \begin{align}
&\|\nabla (u^{\mm{d}},h^{\mm{d}})(t)\|_{m-2} + \sqrt{\varepsilon}\|\nabla^2 (u^{\mm{d}},h^{\mm{d}})\|_{L^2_tH^{m-2}}  \nonumber \\
&\lesssim \varepsilon^{5/4} (\|(v^0,H^0)\|_m^{3/2}+ \|(v^0,H^0)\|_m^{ 7/4})\mbox{ for any }t>0,\label{1.12}
\end{align}
where we have defined that $(u^{\mm{d}},h^{\mm{d}}):=(u,h)-(u^\mm{L},h^\mm{L})$.
\end{thm} 

We mention that Theorem \ref{thma2} can be easily established by following the derivation of \eqref{2504160945} with slight modifications.  
  
\subsection{Results for the original system}

With Theorems \ref{thm1} and \ref{thma2} in hand, we immediately arrive at the following result for the original perturbation system \eqref{1sa.1} by the inverse transformations of \eqref{2511042225}.
\begin{thm}\label{tsafhm1}
Assume that $n=2$, $3$, $\mu>0$, $\nu> 0$ and  
$(v^0,H^0)\in H^m_\sigma$  with the integer $m\geqslant 3$ satisfies \eqref{1.6}, the Cauchy problem of the original perturbation  system \eqref{1sa.1} with the initial data $(v^0,H^0)$ admits a unique global-in-time classical solution $(v,H) \in C^0(\mathbb{R}_0^+,H^m\times H^m)$.
Moreover the solution enjoys the following properties:
\begin{enumerate}[(1)]
\item the stability estimate
\begin{align}
\label{2602142045}
\|(v,H)(t)\|_m +   \|\nabla (v,H)\|_{L^2_t H^{m}} \lesssim  \|(v^0,H^0)\|_m\mbox{ for any }t>0.
\end{align}
\item the vanishing phenomenon of the nonlinear interaction as $\varepsilon\to 0$. More precisely,
  \begin{align}
  \label{2602102130}
\|\nabla (v^{\mm{d}},H^{\mm{d}})(t)\|_{m-2} +\|\nabla^2 (v^{\mm{d}},H^{\mm{d}})\|_{L^2_tH^{m-2}}  \lesssim  \varepsilon^{1/4} (\|(v^0,H^0)\|_m^{3/2}+\|(v^0,H^0)\|_m^{7/4})
\end{align}
for any $t>0$,
where $(v^{\mm{d}},H^{\mm{d}}):=(v-v^{\mm{L}},H-H^{\mm{L}}) $ and $(v^{\mm{L}},H^{\mm{L}}) $ is the unique classical solution to 
 the linear problem 
 \begin{align}
  &\begin{cases}
   v^\mm{L}_t- \mu\Delta v^\mm{L}=\varepsilon^{-1}\partial_1H^\mm{L},\\  H^\mm{L}_t- \nu\Delta H^\mm{L} =\varepsilon^{-1}\partial_1v^\mm{L},\\  
(v^\mm{L}, H^\mm{L})|_{t=0}=(v^0,H^0).
 \end{cases} 
 \label{2602102212}
 \end{align}
 \end{enumerate} 
\end{thm}
\begin{rem}
In particular, if $\|(v^0,H^0) \|_m \geqslant 1 $, the smallness condition \eqref{1.6}
reduces to
\begin{equation*} 
\varepsilon   \|(v^0,H^0) \|_m^3 \leqslant  \delta.
\end{equation*}
Thus Assertion \ref{2510071030} with ``$\mu>0$ and $\nu>0$"  holds for $s=1/3$.    \end{rem}
\begin{rem} We mention that the existence result of large perturbation solutions with small $\varepsilon$ has been provided in \cite{cai2025small} for the Cauchy problem of the system \eqref{1sa.1} with $\mu=\nu>0$ by exploiting an additional regularity condition $\left\||\nabla|^{-1}(v^{0},H^0) \right\|_0<\infty$. However, such additional condition can be removed in the above Theorem \ref{tsafhm1}.
\end{rem}

\subsection{Small Alfv\'en number limit problem}\label{sec2}

Small Alfv\'en number limit (one type of large parameter limits  \cite{MR871107,MR957005} due to $\varpi\to \infty$ as $\varepsilon\to 0$) is one of the distinguished singular limits for MHD system \cite{MR3803773}. For 3D compressible ideal MHD system, the singular problem with respect to the Alfv\'en number in $\mathbb{T}^3$ (i.e. a spatially periodic domain) was first proposed by Klainerman--Majda \cite{MR615627}, in which the small Mach number limit was also investigated. 
However, they cannot prove convergence to an appropriate reduced system. To obtain the limit system,
Browning--Kreiss posed more unnatural assumptions on initial data, i.e., high-order time derivatives
of the solution are assumed to be bounded uniformly with respect to the small Alfv\'en numbers at $t=0$ \cite{MR665380}.

Later, Goto focused on the incompressible case of the 3D ideal MHD system under the natural assumptions on the initial data and first determined that the limiting system becomes essentially the system of the 2D motion \cite{MR1039472}.  Rubino further considered the viscous case \cite{MR1339828}. Recently Jiang--Ju--Xu extended Goto's result to the case that the domain is a slab with horizontal periodicity, and found that the limiting system is a 2D Euler system coupled with a linear transport equation \cite{MR4028271}. 
Similar results can be found in the corresponding compressible case, see \cite{MR3836806,MR3962511,MR4221327,MR4853427,MR4425021} for examples.  
 
It should be emphasized that the above works with a limit system  are only valid for well-prepared initial data. When the initial data is ill-prepared, the fast oscillating waves will be developed and the mathematical analysis is more complicated. 
Recently Ju--Wang--Xu mathematically verified that the \emph{local-in-time} solutions of 3D incompressible  ideal MHD equations with general initial data in $\mathbb{T}^3$ tend to a solution of 2D ideal MHD equations in the distribution sense as $\varepsilon\to 0$ \cite{MR4645635} by Schochet's fast averaging method in \cite{MR1303036}.
However, the small Alfv\'en number limit in 3D \emph{compressible ideal} MHD equations with \emph{general initial data}  is one of the largely wide open problems (see Majda's remark on pp. 72 in his monograph \cite{MR748308}). 
We mention the small Alfv\'en number limit of \emph{global-in-time} weak solutions of 3D non-isentropic compressible \emph{viscous resistive}  MHD equations has been investigated by Kuku$\mathrm{\check{c}}$ka \cite{MR2805860}, but the limit system is still 3D. 

Recently, Cai--Cui--Jiang--Liu revisited the small Alfv\'en number limit in $\mathbb{R}^n$ for the incompressible MHD equations  with  general initial data for the case
${\mu}=\nu \geqslant 0$, and used a energy method with Alinhac's ghost weight technique, to establish the conclusion on small Alfv\'en number limit, i.e. the \emph{global-in-time} classical solutions converge to zero as $\varepsilon\to 0$ for any given time-space variable $(x,t)$ with $t>0$ \cite{cai2025small}. 
The approach of Cai et.al. for the small Alfv\'en number limit  strongly depends on the symmetric structure of equations, i.e. the case ${\mu}=\nu$, and thus fails to the case ${\mu}\neq\nu$.
However, thanks to the error estimate \eqref{2602102130}, ones easily see that the derivation of the small Alfv\'en number limit of the solutions of the nonlinear problem reduces to the one of the solutions of the linear problem. In particular, we have the following small Alfv\'en number limit for the linear problem by analyzing the integral expression of the linear fundamental solution, which has oscillatory structures with respect to $\varepsilon$.
\begin{thm}\label{tsafasdhmdsfa1}
Let $n=2$, $3$, $\mu>0$, $\nu> 0$ and  
$(v^0,H^0)\in H^m_\sigma\cap L^1$  with the integer $m\geqslant 3$.
\begin{enumerate}[(1)]
 \item  We  denote the classical solution of the linear problem \eqref{2602102212} by $(v^\varepsilon_\mm{L},H^\varepsilon_\mm{L})$, thus the family of solutions $\{(v^\varepsilon_\mm{L},H^\varepsilon_\mm{L})\}_{\varepsilon>0}$ satisfies that, for any given $(x,t)\in \mathbb{R}^n\times \mathbb{R}^+$,
\begin{equation*} 
(v^\varepsilon_\mm{L},H^\varepsilon_\mm{L})\to 0 \mbox{ as }\varepsilon \to 0.
\end{equation*}
  \item  Additionally assume that the support of initial data $(v^0, H^0)$ is a subset of the slab $S_R$, where $S_R:=\{x\in \mathbb{R}^n~|~|x_1|<R\}$, thus, for any given $\theta\in (0,1)$, $l\in \mathbb{Z}^+$, and $t_0>0$, the family of solutions $\{v^\varepsilon_\mm{L},H^\varepsilon_\mm{L})\}_{\varepsilon>0}$ of the linear problem   \eqref{2602102212} satisfies that
\begin{equation*} 
\|(v^\varepsilon_\mm{L},H^\varepsilon_\mm{L} )(t)\|_{L^\infty(S_{l R})} \lesssim_{\theta,R, t_0,l}  \varepsilon^\theta \|(v^0, H^0)\|_{L^1}\end{equation*}
 for any $t>t_0$ and for any $\varepsilon\in (0,1)$.
\end{enumerate}
\end{thm}

Exploiting the error estimate \eqref{2602102130}, Theorem \ref{tsafasdhmdsfa1} and the 
well-known Nirenberg interpolation inequality
$$\| f\|_{L^\infty}\lesssim \| f\|_0^{1-n/4}\|\nabla^2 f\|_0^{n/4},$$
 we immediately get the following small Alfv\'en number limit for the solutions of the original perturbation  system \eqref{1sa.1} with the given initial data $(v^0,H^0)$.
\begin{thm}\label{tsafasdhm1}
Let the assumptions in Theorem \ref{tsafhm1} be satisfied. We denote the classical solution $(v,H)$ in Theorems \ref{tsafhm1} by $(v^\varepsilon,H^\varepsilon)$ to emphasize that the solution depends on $\varepsilon$, and further assume the initial data  $(v^0, H^0)\in L^1$.\begin{enumerate}
                      \item[(1)] It holds,  for any given $(x,t)\in \mathbb{R}^n\times \mathbb{R}^+$,
\begin{equation*}
 (v^\varepsilon,H^\varepsilon)\to 0 \mbox{ as }\varepsilon \to 0.
\end{equation*} 
\item[(2)] Additionally assume that the support of initial data $(v^0, H^0)$ is a subset of the slab $S_R$, then it holds that, for any given $t_0>0$ and for any given $l\in \mathbb{Z}^+$,
\begin{align*}
& \|(v^\varepsilon,H^\varepsilon)(t)\|_{L^\infty(S_{l R})}  \\
 &\lesssim_{\theta,R, t_0,l} \varepsilon^{n/16} (\|(v^0,H^0)\|_m^{1+n/8}+\|(v^0,H^0)\|_m^{1+3n/16}+\|(v^0, H^0)\|_{L^1}) 
\end{align*}
 for any $t>t_0$ and for any $\varepsilon\in (0,1)$.
                    \end{enumerate}
\end{thm}

The rest three sections are devoted to the proofs of Theorems \ref{thm1}, \ref{thma2} and \ref{tsafasdhmdsfa1} resp..
   
\section{Proof of Theorem \ref{thm1}} 
As mentioned before, the key proof of Theorem \ref{thm1} is to derive the \emph{a priori} stability estimate \eqref{2504160945} of solutions for the Cauchy problem \eqref{1.4} defined on a time interval $[0,T]$.

\subsection{A basic energy estimate}

To being with, we derive the basic energy estimate for the smooth solution $(u,h)$ of the Cauchy problem \eqref{1.4} defined on $[0,T]$.
\begin{lem}\label{2511041741}
It holds that
\begin{align}
\label{2510071saf637}
&  \|(u,  h)(t)\|_{m}^2 
+ \varepsilon\|\nabla  (u, h)\|_{L^2_tH^{m}}^2 \nonumber 
\\
& \lesssim  \varepsilon^2  \|(v^0,H^0)\|_m^2+\left\|\| \nabla \Lambda_+\|_{H^{m-2}_{{x}_{\mm{v}}}} \| \nabla \Lambda_-\|_{H^{m-1}_{{x}_{\mm{v}}}}\right.\nonumber \\
&\left.\quad  +\| \nabla \Lambda_-\|_{H^{m-2}_{{x}_{\mm{v}}}} \| \nabla \Lambda_+\|_{H^{m-1}_{{x}_{\mm{v}}}} \right\|_{L^2_{t,x_1}} \|  \nabla (u, h)\|_{L^2_tH^{m}}\mbox{ for any }t\in (0,T].
\end{align}
\end{lem}
\begin{pf}
Let the multi-index $\alpha$ satisfy $0<|\alpha|\leqslant m$ and $\alpha_i\neq 0$ for some $i\in \{1,\ldots, n\}$. We further define $\beta$ as follows
$$\beta_j=\alpha_j\mbox{ for }j\neq i\mbox{ and } \beta_i=\alpha_i-1.$$
In addition, $$C^\alpha_{\alpha-\gamma}:= \left(
              \begin{array}{c}
                    \alpha\\
                    \alpha-\gamma
                 \end{array}
               \right):=\frac{ \alpha!}{\gamma!(\alpha-\gamma)!}.$$

Applying $\partial^\alpha$ to \eqref{1.4} yields 
\begin{equation}\label{1safa.4}
\begin{cases}
\partial^\alpha(u_t+  u\cdot\nabla u- \varepsilon {\mu}\Delta u+\nabla q )=\partial^\alpha(( \mathbf{e}^1+ h )\cdot \nabla h), \\[1mm]
\partial^\alpha(h_t+u\cdot\nabla h-\varepsilon\nu\Delta h)=\partial^\alpha((\mf{e}^{1} +h)\cdot\nabla u).
\end{cases}
\end{equation}
Taking the inner products of \eqref{1safa.4}$_1$, resp. \eqref{1safa.4}$_2$ and $\partial^\alpha u$ resp. $\partial^\alpha h$ in $L^2$, then summing the two resulting identities together, and finally using the integration by parts and the divergence-free condition in \eqref{1.4}$_3$, we obtain that
\begin{align}
&\frac{1}{2}\frac{\rm d}{{\rm d}t} \left( \|\partial^\alpha  u\|_{0}^2 + \|\partial^\alpha  h\|_{0}^2 \right)
+ \varepsilon(\mu \|\nabla \partial^\alpha  u\|_{0}^2 + \nu \|\nabla \partial^\alpha  h\|_{0}^2)\nonumber  \\
= &\sum_{ \gamma \leqslant\alpha  } C^\alpha_{\alpha-\gamma}\int\left(\left( \partial^\gamma h\cdot\nabla \partial^{\alpha-\gamma } h- \partial^\gamma u\cdot\nabla \partial^{\alpha-\gamma } u \right) \cdot\partial^\alpha u \right.\nonumber  \\
& \left.\vphantom{} +  \left(\partial^\gamma  h\cdot\nabla \partial^{\alpha-\gamma } u- \partial^\gamma  u\cdot\nabla \partial^{\alpha-\gamma } h  \right) \cdot\partial^\alpha h\right){\rm d}x\nonumber \\ 
= & \sum_{ \gamma \leqslant\alpha  }  \frac{C^\alpha_{\alpha-\gamma }}{2}\int
(( 
\partial^\gamma  \Lambda_+\cdot\nabla \partial^{\alpha-\gamma }\Lambda_--\partial^\gamma  \Lambda_-\cdot\nabla \partial^{\alpha-\gamma }\Lambda_+)\cdot \partial^\alpha h
 \nonumber \\
& -(
\partial^\gamma  \Lambda_+\cdot\nabla \partial^{\alpha-\gamma }\Lambda_-+\partial^\gamma  \Lambda_-\cdot\nabla \partial^{\alpha-\gamma }\Lambda_+)\cdot \partial^\alpha u){\rm d}x \nonumber \\
 = &- \sum_{ \gamma \leqslant\alpha  }  \frac{C^\alpha_{\alpha-\gamma }}{2}\int
\partial^\gamma  \Lambda_+\cdot\nabla \partial^{\alpha-\gamma }\Lambda_-\cdot \partial^\alpha \Lambda_-
 {\rm d}x\nonumber  \\ 
 &-\sum_{ \gamma \leqslant\alpha  }  \frac{C^\alpha_{\alpha-\gamma }}{2}\int\partial^\gamma  \Lambda_-\cdot\nabla \partial^{\alpha-\gamma }\Lambda_+\cdot \partial^\alpha\Lambda_+ {\rm d}x=:I_1+I_2. \label{2510081145}
\end{align}
Next we estimate for $I_1$ and $I_2$ in sequence.

Exploiting both embedding inequalities of 
$$H^2(\mathbb{R}^{n-1})\hookrightarrow L^\infty(\mathbb{R}^{n-1})\mbox{ and }H^1(\mathbb{R}^{n-1})\hookrightarrow L^4(\mathbb{R}^{n-1}),$$ the divergence-free conditions of $\Lambda_\pm$,  H\"older's inequality, and the integration by parts, we deduce that 
 \begin{align}
I_1=&-\sum_{\{\gamma \leqslant\alpha\,|\,0<|\gamma |\leqslant m-2\}} \frac{C^\alpha_{\alpha-\gamma }}{2} \int \partial^\gamma  \Lambda_+\cdot\nabla \partial^{\alpha-\gamma }\Lambda_-  \cdot \partial^\alpha \Lambda_-{\rm d}x\nonumber \\
&-\sum_{\{\gamma \leqslant\alpha\,|\, |\gamma |=m-1 \} } \frac{C^\alpha_{\alpha-\gamma }}{2} \int \partial^\gamma  \Lambda_+\cdot\nabla \partial^{\alpha-\gamma }\Lambda_-  \cdot \partial^\alpha \Lambda_-{\rm d}x\nonumber \\
&\begin{cases}
\displaystyle -\frac{1}{2}\int \partial^\alpha \Lambda_+\cdot\nabla  \Lambda_-  \cdot \partial^\alpha \Lambda_-{\rm d}x&\mbox{for }|\alpha|=1;\\
\displaystyle+\frac{1}{2}\int \partial^{\beta } \Lambda_+\cdot\partial_i(\nabla  \Lambda_-  \cdot \partial^\alpha \Lambda_-){\rm d}x&\mbox{for }|\alpha|>1
 \end{cases}\nonumber \\
 \lesssim & \sum_{\{\gamma \leqslant\alpha \,|\, 0<|\gamma |\leqslant m-2\}} \int_{\mathbb{R}}\|\partial^\gamma \Lambda_+\|_{L^4_{{x}_\mm{v}}}\|\partial^\alpha   \Lambda_-\|_{L^2_{{x}_\mm{v}}}\|\nabla \partial^{\alpha-\gamma }\Lambda_-\|_{L^4_{{x}_\mm{v}}}{\rm d}x_1 \nonumber\\ 
 &+\sum_{\{\gamma \leqslant\alpha\,|\,|\gamma |=m-1  \}}  \int_{\mathbb{R}}\|\partial^\gamma \Lambda_+\|_{L^2_{{x}_\mm{v}}}\|\nabla \partial^{\alpha-\gamma }\Lambda_-\|_{L^4_{{x}_\mm{v}}}\|\partial^\alpha   \Lambda_-\|_{L^4_{{x}_\mm{v}}}{\rm d}x_1 \nonumber\\
 &+  \begin{cases}
\displaystyle \int_{\mathbb{R}}\|\partial^\alpha \Lambda_+\|_{L^2_{{x}_\mm{v}}}\|\nabla \Lambda_-\|_{L^\infty_{{x}_\mm{v}}}\|\partial^\alpha   \Lambda_-\|_{L^2_{{x}_\mm{v}}}{\rm d}x_1&\mbox{for }|\alpha|=1,\\
\displaystyle \int_{\mathbb{R}}\|\partial^\beta \Lambda_+\|_{L^2_{{x}_\mm{v}}}(\|\nabla \Lambda_-\|_{L^\infty_{{x}_\mm{v}}}\|\partial_i \partial^\alpha   \Lambda_-\|_{L^2_{{x}_\mm{v}}}&\\
+\|\partial_i \nabla \Lambda_-\|_{L^4_{{x}_\mm{v}}}\|\partial^\alpha   \Lambda_-\|_{L^4_{{x}_\mm{v}}} ){\rm d}x_1&\mbox{for }|\alpha|>1
 \end{cases}\nonumber \\
\lesssim  &   \left\|\| \nabla \Lambda_+\|_{H^{m-2}_{{x}_\mm{v}}} \| \nabla \Lambda_-\|_{H^{m-1}_{{x}_\mm{v}}}\right\|_{L^2_{x_1}}\| \nabla (u, h)\|_{{m}}.
\label{2511031920}
\end{align} 

Similarly to \eqref{2511031920}, we also have
\begin{align*}
I_{2}& \leqslant\left\|\| \nabla \Lambda_-\|_{H^{m-2}_{{x}_{\mm{v}}}} \| \nabla \Lambda_+\|_{H^{m-1}_{{x}_{\mm{v}}}}\right\|_{L^2_{x_1}}\|\nabla (u, h)\|_{m} . 
\end{align*}
Consequently, inserting the above two estimates into \eqref{2510081145}, then integrating the resulting inequality over $(0,t)$, and finally using the energy identity \eqref{2511041244}, we arrive at \eqref{2510071saf637}.
\hfill$\Box$
\end{pf}

\begin{lem}
It holds that
\begin{align}
\label{2510081449}
\|\nabla q\|_{i}\lesssim \|h\|_{i}\|\nabla h\|_{i}+\|u\|_{i}\|\nabla u\|_{i}\mbox{ for any }2 \leqslant i\leqslant m .
\end{align}
\end{lem}
\begin{pf}
Let $\alpha$ satisfy $0\leqslant |\alpha|\leqslant i$.
Applying $\mm{div}$ to \eqref{1.4}$_1$ and using the divergence-free condition of $u$ yields 
 \begin{align*} 
 \partial^\alpha  \Delta q = \partial^\alpha\mm{div}( h \cdot \nabla h -u\cdot\nabla u).
\end{align*}
Taking the inner product of the above identity and $ \partial^\alpha    q$ in $L^2$, and then using 
H\"older's inequality, the integration by parts, the well-known product estimate
\begin{align}
\|\phi\psi\|_{j} \lesssim \|\phi\|_{j}\|\psi\|_2+\|\phi\|_2\|\psi\|_{j}\mbox{ for any }{j}\geqslant 2,
\label{2511052109}
\end{align}  we obtain that
 \begin{align*} 
 \|\partial^\alpha  \nabla q \|_0\lesssim \| \partial^\alpha( h \cdot \nabla h -u\cdot\nabla u)\|_0\lesssim \|h\|_{i}\|\nabla h\|_{i}+\|u\|_{i}\|\nabla u\|_{i},
\end{align*}
which yields \eqref{2510081449}. The proof is complete. 
\hfill$\Box$
\end{pf}
\subsection{A bilinear estimate}
This section is devoted to the derivation of a bilinear estimate. To begin with, we shall introduce the the following div-curl lemma.
\begin{lem}[Div-curl lemma, see Lemma 2.3 in \cite{lai2025global}]
\label{2510081319}
Suppose that $f_{ij}(y,t)$ is defined on $\mathbb{R}\times \mathbb{R}^+_0$ and satisfies
\begin{align*}
\begin{cases}
\partial_t f_{11} + \partial_y f_{12} = G_1, \\
\partial_t f_{21} - \partial_y f_{22} = G_2
\end{cases}
\end{align*}
and $f_{ij} \to 0$ as $y \to \infty$ for $i$, $j\in \{1, 2\}$, then it holds that
\begin{align*}
&\int_0^T \int_{\mathbb{R}} (f_{11} f_{22} + f_{12} f_{21}) {\rm d}y{\rm d}t \nonumber\\
&\leqslant 2 \sup_{0 \leqslant t \leqslant T} \{\|f_{11}(t)\|_{L^1(\mathbb{R})} \|f_{21}(t)\|_{L^1(\mathbb{R})}\} + \left|\int_0^T \int_{\mathbb{R}} \left(\int_{-\infty}^y f_{11}(z,t) {\rm d}z\right) G_2(y,t) {\rm d}y{\rm d}t\right|  \nonumber \\
&\quad + \left|\int_0^T \int_{\mathbb{R}} \left(\int_y^{+\infty} f_{21}(z,t) {\rm d}z\right) G_1(y,t) {\rm d}y{\rm d}t\right|,
\end{align*}
provided that the right side is bounded.
\end{lem}
Next we establish the following bilinear estimate.
\begin{lem}\label{2511041742}
It holds that
\begin{align}
\label{2510071241}
& \left\|\| \Lambda_+\|_{H^{m-1}_{{x}_{\mm{v}}}}\| \nabla \Lambda_-
\|_{H^{m-1}_{{x}_{\mm{v}}}}+\|\Lambda_-
\|_{H^{m-1}_{{x}_{\mm{v}}}}\| \nabla \Lambda_+\|_{H^{m-1}_{{x}_{\mm{v}}}}\right\|_{L^2_{T,x_1}} \nonumber 
\\
&\lesssim
\sup_{0 \leqslant t \leqslant T}\{\|  (u,h) (t)\|_{m}^2\}\nonumber \\
&\qquad  + 
 \sup_{0 \leqslant t \leqslant T}\{\|  (u,h) (t)\|_{m}\}  \left(\sup_{0 \leqslant t \leqslant T}\{\|  (u,h) (t)\|_{m}\}
 +{\varepsilon}\right)^{1/2}\| \nabla (u,h) \|_{L^2_TH^{m}}.
\end{align}
\end{lem}
\begin{pf}
Let multi-indexes $\alpha$ and $\gamma $ satisfy $ |\alpha|\leqslant m$, $ |\gamma |\leqslant m-1$.
Taking the inner products of \eqref{1safa.4}$_1$ resp. \eqref{1safa.4}$_2$  and $\partial^\alpha  u$ resp. 
$\partial^\alpha  h $ in $\mathbb{R}^n$ for given $(x,t)$, and then adding the two resulting identities together, we have
\begin{align}
& \partial_t\left( |\partial^\alpha  u|^2+
|\partial^\alpha  h|^2\right)-  2\partial_1(\partial^\alpha u \cdot \partial^\alpha  h)   = 2\tilde{G}_{1},
\label{25010071750}
\end{align}
where we have defined that
\begin{align*}\tilde{G}_{1}:=&\partial^\alpha( h\cdot \nabla h-u\cdot\nabla u-\nabla q)\cdot\partial^\alpha u+
\partial^\alpha (h\cdot \nabla u -u\cdot \nabla h) \cdot \partial^\alpha  h \\
&  + {\varepsilon}\left(\mu\Delta\partial^\alpha u \cdot \partial^\alpha u+
\nu\Delta\partial^\alpha  h\cdot\partial^\alpha h\right).
\end{align*}

Let 
\begin{align*}
f_{11} = \frac{1}{2}\int_{\mathbb{R}^{n-1}} ({|\partial^\alpha  u|^2 + |\partial^\alpha  h|^2}) {\rm d}{x}_{\mm{v}}\mbox{ and }
f_{12} = - \int_{\mathbb{R}^{n-1}}  
 \partial^\alpha  u \cdot\partial^\alpha  h {\rm d}{x}_{\mm{v}} .
\end{align*}
Integrating the identity \eqref{25010071750} with respect to ${x}_{\mm{v}}$ over $\mathbb{R}^{n-1}$, we get
\begin{align}
&\partial_t f_{11}+ \partial_1f_{12} =G_{1}:=\int_{\mathbb{R}^{n-1}}\tilde{G}_{1}{\rm d}{x}_{\mm{v}}.
\label{2510081920}
\end{align}

Let
\begin{align*}
&f_{21} =\int_{\mathbb{R}^{n-1}} \partial^\gamma  u \cdot \partial^\gamma  h {\rm d}{x}_{\mm{v}}\mbox{ and }f_{22} = \frac{1}{2} \int_{\mathbb{R}^{n-1}} (|\partial^\gamma  u|^2 + |\partial^\gamma  h|^2) {\rm d}{x}_{\mm{v}},\\
&\tilde{G}_{2}=  \partial^\gamma ( h\cdot \nabla h-u\cdot\nabla u-\nabla q)\cdot\partial^\gamma h+
\partial^\gamma (h\cdot \nabla u -u\cdot \nabla h) \cdot \partial^\gamma  u \\
&\qquad   + {\varepsilon}\left(\mu \Delta \partial^\gamma u\cdot \partial^\gamma h+\nu \Delta\partial^\gamma  h\cdot\partial^\gamma u\right).
\end{align*}
 Taking the inner products of \eqref{1safa.4}$_1$, resp. \eqref{1safa.4}$_2$ with $\partial^\gamma  h$ resp. $\partial^\gamma  u$ in $L^2(\mathbb{R}^{n-1})$ for any given $(x_1,t)$,  then adding the two resulting identities together, and finally replacing $\alpha$ by $\gamma $, we get
\begin{align}
\label{25100819201}
&\partial_tf_{21} - \partial_1 f_{22}=G_2:=\int_{\mathbb{R}^{n-1}}\tilde{G}_{2}{\rm d}{x}_{\mm{v}} .
\end{align}

By a simple computation, we have
\begin{align*}
& \|\partial^\gamma  \Lambda_+\|^2_{L^2_{{x}_{\mm{v}}}} \|\partial^\alpha  \Lambda_-\|^2_{L^2_{{x}_{\mm{v}}}}
+  \|\partial^\gamma  \Lambda_-\|^2_{L^2_{{x}_{\mm{v}}}}\|\partial^\alpha  \Lambda_+\|^2_{L^2_{{x}_{\mm{v}}}}\\
=&2\int_{\mathbb{R}^{n-1}} ({|\partial^\alpha  u|^2 + |\partial^\alpha  h|^2}){\rm d}{x}_{\mm{v}}\int_{\mathbb{R}^{n-1}} ({|\partial^\gamma  h|^2 + |\partial^\gamma  u|^2}) {\rm d}{x}_{\mm{v}}\nonumber \\
& -8 \int_{\mathbb{R}^{n-1}} \partial^\alpha  u \cdot \partial^\alpha  h {\rm d}{x}_{\mm{v}}\int_{\mathbb{R}^{n-1}} \partial^\gamma  u \cdot \partial^\gamma  h {\rm d}{x}_{\mm{v}}. \end{align*}
Using the above identity, we  further compute out that 
\begin{align*}
&f_{11}f_{22} + f_{12}f_{21} \\
&= \frac{1}{4}\int_{\mathbb{R}^{n-1}}(|\partial^\alpha  u|^2 + |\partial^\alpha  h|^2) {\rm d}{x}_{\mm{v}} \int_{\mathbb{R}^{n-1}} (|\partial^\gamma  h|^2 + |\partial^\gamma  u|^2){\rm d}{x}_{\mm{v}}
  \nonumber \\
&\quad  -\int_{\mathbb{R}^{n-1}} \partial^\alpha  u \cdot \partial^\alpha  h {\rm d}{x}_{\mm{v}}\int_{\mathbb{R}^{n-1}} \partial^\gamma  u \cdot \partial^\gamma  h{\rm d}{x}_{\mm{v}}\\ 
& =\frac{1}{8}\left(\|\partial^\gamma  \Lambda_-\|^2_{L^2_{{x}_{v}}}\|\partial^\alpha  \Lambda_+\|^2_{L^2_{{x}_{v}}} + \|\partial^\gamma  \Lambda_+\|^2_{L^2_{{x}_{v}}}\|\partial^\alpha  \Lambda_-\|^2_{L^2_{{x}_{v}}} \right) .
\end{align*} 

Applying Lemma \ref{2510081319} to \eqref{2510081920} and \eqref{25100819201}, and then using the above identity, we obtain
\begin{align}
 & \frac{1}{16}\left(\| \|\partial^\gamma  \Lambda_+\|_{L^2_{{x}_{\mm{v}}}} \|\partial^\alpha  \Lambda_-
 \|_{L^2_{{x}_{\mm{v}}}} + \|\partial^\gamma  \Lambda_-\|_{L^2_{x_{\mm{v}}}}\|\partial^\alpha  \Lambda_+\|_{L^2_{x_{\mm{v}}}} \| 
^2_{L^2_TL^2_{x_1}} \right)\nonumber \\ \label{20251saf0062217}
&\leqslant \int_0^T \int_{\mathbb{R}} (f_{11} f_{22} + f_{12} f_{21}) {\rm d}x_1{\rm d}t  \leqslant  \sum_{j=1}^{3}|I_{2+j}|,
\end{align}
where we have defined that 
\begin{align}
&I_{3}:= 2\sup_{0 \leqslant t \leqslant T} \{\|f_{11}(t)\|_{L^1(\mathbb{R})} \|f_{21}(t)\|_{L^1 (\mathbb{R})}\}, \nonumber \\ 
& I_{4}:=\int_0^T \int_{\mathbb{R}} \left(\int_{-\infty}^{x_1} f_{11}(s,t) {\rm d}s\right) G_2(x_1,t){\rm d}x_1{\rm d}t,   \nonumber \\
& I_{5}:= \int_0^T \int_{\mathbb{R}} \left(\int_{x_1}^{+\infty} f_{21}(s,t) {\rm d}s\right) G_1(x_1,t){\rm d}x_1{\rm d}t. \nonumber 
\end{align}
Next we estimate for the above three integrals in sequence.
 
(1) Using H\"older's inequality, it's obvious that
\begin{align}
I_3\lesssim &  \sup_{0 \leqslant t \leqslant T}\{\| \partial^\alpha (u,h) (t)\|_{0}^2  \| \partial^\gamma  h(t)\|_{0} \| \partial^\gamma  u(t) \|_{0}\}.
\label{25100082108}
\end{align}
 
(2)   We split $I_4$ into two terms
\begin{align}\label{251002513}
 I_4=I_{4,1}+ {\varepsilon} I_{4,2}  
 , \end{align}
 where we have defined that 
\begin{align*}
I_{4,1}:= &\frac{1}{2}\int_0^T \int_{\mathbb{R}} \left(\int_{-\infty}^{x_1} \int_{\mathbb{R}^{n-1}} ({|\partial^\alpha  u(s,x_{\mm{v}},t)|^2 + |\partial^\alpha  h(s,x_{\mm{v}},t)|^2}) {\rm d}x_{\mm{v}}  {\rm d}s\right) 
 \int_{\mathbb{R}^{n-1}}(\partial^\gamma ( h\cdot \nabla h\\
  & -u\cdot\nabla u-\nabla q)\cdot\partial^\gamma h +
\partial^\gamma (h\cdot \nabla u -u\cdot \nabla h) \cdot \partial^\gamma u
 ){\rm d}x_{\mm{v}}{\rm d}x_1{\rm d}t 
 \end{align*}
 and 
\begin{align*}
 I_{4,2}:=&\frac{1}{2} \int_0^T \int_{\mathbb{R}} \left(\int_{-\infty}^{x_1} \int_{\mathbb{R}^{n-1}} ({|\partial^\alpha  u(s,x_{\mm{v}},t)|^2 + |\partial^\alpha  h(s,x_{\mm{v}},t)|^2}) {\rm d}x_{\mm{v}}  {\rm d}s\right)\int_{\mathbb{R}^{n-1}} (\mu \Delta \partial^\gamma u\cdot \partial^\gamma h\\
&+\nu \Delta\partial^\gamma  h\cdot\partial^\gamma u ){\rm d}x_{\mm{v}}   {\rm d}x_1{\rm d}t .\end{align*}

From now on, we assume $ |\alpha| >0$. 
Exploiting \eqref{2510081449} and the product estimate \eqref{2511052109},
we have
\begin{align}
|I_{4,1}|\lesssim  &\sup_{0\leqslant t\leqslant T} \{\|\partial^\alpha  (u,h)(t)\|_0 \|\partial^\gamma (u,h)(t)\|_0\}
\int_0^T \|\partial^\alpha  (u,h)\|_0(\|\partial^\gamma ( h\cdot \nabla h   \nonumber \\
&-u\cdot\nabla u-\nabla q) \|_0+
\|\partial^\gamma (h\cdot \nabla u -u\cdot \nabla h)\|_0)\mm{d}t\nonumber \\
 \lesssim  & \sup_{0 \leqslant t \leqslant T}\{\|  (u,h)(t) \|_{m}^3\}  \| \nabla (u,h) \|_{L^2_TH^{m}}^2.
\label{2602262358}
 \end{align}
Similarly, \begin{align*}
|I_{4,2}| 
\lesssim  & \sup_{0 \leqslant t \leqslant T}\{\|  (u,h) (t)\|_{m}^2\}
 \| \nabla (u,h) \|_{L^2_TH^{m}}^2 .
 \end{align*}
Thanks to the above two estimates, we deduce from \eqref{251002513} that
\begin{align}
&|I_4|\lesssim \sup_{0 \leqslant t \leqslant T}\{\|  (u,h)(t) \|_{m}^2\}
\left(  \sup_{0 \leqslant t \leqslant T}\{\|  (u,h) (t)\|_{m}\} + {\varepsilon}\right) \| \nabla (u,h) \|_{L^2_TH^{m}}^2 . 
\label{2510082113}
\end{align} 

(3) Similarly to \eqref{2602262358}, 
 we also split $I_5$ into two terms
\begin{align}\label{2510025131}
 I_5=I_{5,1}+ {\varepsilon} I_{5,2}  
 , \end{align}
 where we have defined that 
\begin{align*}
I_{5,1}:= & \int_0^T \int_{\mathbb{R}} \left(\int_{x_1}^{+\infty} \int_{\mathbb{R}^{n-1}}(\partial^\gamma u\cdot \partial^\gamma h)(s,x_{\mm{v}},t){\rm d}x_{\mm{v}}{\rm d}s\right) 
 \int_{\mathbb{R}^{n-1}}(\partial^\alpha( h\cdot \nabla h \\
& -u\cdot\nabla u-\nabla q)\cdot\partial^\alpha u +
\partial^\alpha (h\cdot \nabla u -u\cdot \nabla h) \cdot \partial^\alpha  h ) {\rm d}x_{\mm{v}}{\rm d}x_1{\rm d}t 
\end{align*}
and 
\begin{align*}
I_{5,2}:=& \int_0^T \int_{\mathbb{R}} \left(\int_{x_1}^{\infty} \int_{\mathbb{R}^{n-1}}(\partial^\gamma u\cdot \partial^\gamma h)(s,x_{\mm{v}},t) {\rm d}x_{\mm{v}}  {\rm d}s\right)\int_{\mathbb{R}^{n-1}} (\mu\Delta\partial^\alpha u \cdot \partial^\alpha u\\
&+\nu\Delta\partial^\alpha  h\cdot\partial^\alpha h ){\rm d}x_{\mm{v}}   {\rm d}x_1{\rm d}t.\end{align*}
Following the argument of \eqref{2602262358}, we also have
 \begin{align}
|I_{5,1}| 
 \lesssim  \sup_{0 \leqslant t \leqslant T}\{\|  (u,h)(t) \|_{m}^3\} 
   \| \nabla (u,h) \|_{L^2_TH^{m}}^2  .
\label{2510082115fsa}
\end{align} 

 It is easy to check that
\begin{align}
\label{2511041525}
\|\|\phi\|_{L^2_{x_\mm{v}}}\|_{L^\infty_{x_1}}\lesssim \|\|\phi\|_{L^\infty_{x_1}} \|_{L^2_{x_\mm{v}}} \lesssim \|(\|\phi\|_{L^2_{x_1}}\|\partial_1 \phi\|_{L^2_{x_1}})^{1/2} \|_{L^2_{x_\mm{v}}}\lesssim (\|\phi\|_{0}\|\phi \|_1)^{1/2},
\end{align}
provided that the right side is bounded.
Exploiting \eqref{2511041525} and the integration by parts, we easily estimate that
\begin{align}
|I_{5,2}|\le & \ \mu \left|\int_0^T \int_{\mathbb{R}}
 \int_{\mathbb{R}^{n-1}} \left( \int_{x_1}^{+\infty} \int_{\mathbb{R}^{n-1}}(\partial^\gamma u\cdot \partial^\gamma h)(s,y_{\mm{v}},t){\rm d}y_{\mm{v}}{\rm d}s \, \partial^\alpha \nabla u:\partial^\alpha\nabla  u \right.\right. \nonumber \\
&\left.\left.- \int_{\mathbb{R}^{n-1}} (\partial^\gamma u\cdot \partial^\gamma h)(x_1, y_{\mm{v}}, t)  {\rm d}y_{\mm{v}} \,
\partial_1 \partial^\alpha u \cdot \partial^\alpha  u\right)  {\rm d}x_{\mm{v}}{\rm d}x_1{\rm d}t \right| \nonumber \\
& + \nu \left|\int_0^T \int_{\mathbb{R}}
 \int_{\mathbb{R}^{n-1}} \left( \int_{x_1}^{+\infty} \int_{\mathbb{R}^{n-1}}(\partial^\gamma u\cdot \partial^\gamma h)(s,y_{\mm{v}},t){\rm d}y_{\mm{v}}{\rm d}s \, \partial^\alpha \nabla h:\partial^\alpha\nabla  h \right.\right. \nonumber \\
&\left.\left.- \int_{\mathbb{R}^{n-1}} (\partial^\gamma u\cdot \partial^\gamma h)(x_1, y_{\mm{v}}, t)  {\rm d}y_{\mm{v}} \,
\partial_1 \partial^\alpha h \cdot \partial^\alpha  h\right)  {\rm d}x_{\mm{v}}{\rm d}x_1{\rm d}t \right| \nonumber \\
\lesssim &\left(\sup_{0 \leqslant t \leqslant T}\{\|  (u,h)(t) \|_{m}^2\} + 
\sup_{0 \leqslant t \leqslant T}\{\|\|\partial^\gamma u(t)\|_{L^2_{x_\mm{v}}}\|\partial^\gamma h(t)\|_{L^2_{x_\mm{v}}}\|_{L^\infty_{x_1}}\} \right)\| \nabla (u, h) \|_{L^2_TH^{m}}^2\nonumber \\
 \lesssim &\sup_{0 \leqslant t \leqslant T}\{\|  (u,h)(t) \|_{m}^2\} 
   \| \nabla (u, h) \|_{L^2_TH^{m}}^2  .
\label{2510082115fsa1}
\end{align}   
Thanks to the two estimates \eqref{2510082115fsa} and \eqref{2510082115fsa1}, we deduce from \eqref{2510025131} that 
 \begin{align}
&|I_5|\lesssim \sup_{0 \leqslant t \leqslant T}\{\|  (u,h) (t)\|_{m}^2\}
\left(  \sup_{0 \leqslant t \leqslant T}\{\|  (u,h) (t)\|_{m}\}
+  {\varepsilon} \right)\| \nabla (u,h) \|_{L^2_TH^{m}}^2.
\label{25100dsfa82115}
\end{align} 
Finally, putting \eqref{25100082108}, \eqref{2510082113} and \eqref{25100dsfa82115} into \eqref{20251saf0062217} yields \eqref{2510071241}. The proof is complete. \hfill$\Box$
\end{pf}

\subsection{A stability estimate} 

It follows from Lemmas \ref{2511041741} and \ref{2511041742} that 
\begin{align*} 
& \sup_{0 \leqslant t \leqslant T}\{\|(u, h)(t)\|_{m}^2\} 
+ \varepsilon\| \nabla (u,h) \|_{L^2_TH^{m}}^2 \nonumber 
\\
& \lesssim \varepsilon^2  \|(v^0,H^0)\|_0^2+\left\|\| \nabla \Lambda_+\|_{H^{m-2}_{{x}_{\mm{v}}}} \| \nabla \Lambda_-\|_{H^{m-1}_{{x}_{\mm{v}}}}\right.\nonumber \\
&\left.\qquad + \| \nabla \Lambda_-\|_{H^{m-2}_{{x}_{\mm{v}}}} \| \nabla \Lambda_+\|_{H^{m-1}_{{x}_{\mm{v}}}} \right\|_{L^2_{T,x_1}} \|  \nabla (u, h)\|_{L^2_TH^{m-1}}\nonumber 
\\
&\lesssim \varepsilon^2  \|(v^0,H^0)\|_0^2 +  \sup_{0 \leqslant t \leqslant T}\{\|  (u,h)(t) \|_{m}\}\| \nabla (u,h) \|_{L^2_TH^{m}}
\Bigg( \sup_{0 \leqslant t \leqslant T}\{\|  (u,h)(t) \|_{m}\} \\
&\quad + \left(\sup_{0 \leqslant t \leqslant T}\{\|  (u,h) \|_{m}\}+
 {\varepsilon} \right)^{1/2}\| \nabla (u,h) \|_{L^2_TH^{m}}\Bigg) .
\end{align*}
In particular,  under the \emph{a priori} assumptions \eqref{2510081829} and \eqref{2511042007}, we deduce \eqref{11061426} from the above estimate.  

Based on the stability estimate \eqref{11061426} under the \emph{a priori} assumptions \eqref{2510081829} and \eqref{2511042007}, we can easily extend the 
unique local solutions of the Cauchy problem \eqref{1.4} to the unique global ones, and thus obtain Theorem \ref{thm1}, please refer to \cite{MR4641656} for the extension of solutions with respect to time. This completes the proof of Theorem \ref{thm1}.

\section{Proof of Theorem \ref{thma2}}\label{sec3}

This section is devoted to the proof of Theorem \ref{thma2}. Let $(u,h)$ be the global solution in Theorem \ref{thm1}, and 
$(u^\mm{L},h^\mm{L})$ the unique classical solution to the linear  problem \eqref{1.11} stated in Theorem \ref{thma2}, and  $(u^{\mm{d}},h^{\mm{d}})=(u,h)-(u^\mm{L},h^\mm{L})$. Then $(u^{\mm{d}},h^{\mm{d}})$ satisfies the following error problem
\begin{equation}\label{1asdf.4}
\begin{cases}
u_t^{\mm{d}}- \varepsilon {\mu}\Delta u^{\mm{d}}+\nabla q =\partial_1 h^{\mm{d}}+ h  \cdot \nabla h-   u\cdot\nabla u , \\[1mm]
h_t^{\mm{d}}-\varepsilon\nu\Delta h^{\mm{d}}=  \partial_1v^{\mm{d}}+h\cdot\nabla u-u\cdot\nabla h,\\[1mm]
\mathrm{div}u^{\mm{d}}=\mathrm{div}h^{\mm{d}}=0,\\[1mm]
u^{\mm{d}}|_{t=0}=h^{\mm{d}}|_{t=0}=0.
\end{cases}
\end{equation}
In addition, we easily see from the linear problem \eqref{1.11} that
\begin{align} 
\sup_{0\leqslant t<\infty} \|(u^\mm{L},h^\mm{L})(t)\|_m + \sqrt{\varepsilon}\|\nabla (u^\mm{L},h^\mm{L})\|_{L^2_\infty H^{m}} \lesssim \varepsilon \|(v^0,H^0)\|_m. \label{2511061414}
\end{align} 

Following the argument of \eqref{2510081145}, we deduce from \eqref{1asdf.4} that
\begin{align}
&\frac{1}{2}\frac{\rm d}{{\rm d}t} \left( \|\partial^\alpha  u^{\mm{d}}\|_{0}^2 + \|\partial^\alpha  h^{\mm{d}}\|_{0}^2 \right)
+ \varepsilon(\mu \|\nabla \partial^\alpha  u\|_{0}^2 + \nu \|\nabla \partial^\alpha  h\|_{0}^2)\nonumber  \\ 
= &  \sum_{ \gamma \leqslant\alpha  } \frac{C^\alpha_{\alpha-\gamma }}{2}\int
(( 
\partial^\gamma  \Lambda_+\cdot\nabla \partial^{\alpha-\gamma }\Lambda_--\partial^\gamma  \Lambda_-\cdot\nabla \partial^{\alpha-\gamma }\Lambda_+)\cdot \partial^\alpha h^{\mm{d}}
 \nonumber \\
& -(
\partial^\gamma  \Lambda_+\cdot\nabla \partial^{\alpha-\gamma }\Lambda_-+\partial^\gamma  \Lambda_-\cdot\nabla \partial^{\alpha-\gamma }\Lambda_+)\cdot \partial^\alpha u^{\mm{d}}){\rm d}x =:I_6, \label{25100811saf45}
\end{align}
where $0<|\alpha|< m$.

Following the argument of \eqref{2511031920} with slight modification, we can estiamte that 
 \begin{align*}
I_6\lesssim    \left\|\|  \Lambda_+\|_{H^{m-1}_{{x}_\mm{v}}} \| \nabla \Lambda_-\|_{H^{m-1}_{{x}_\mm{v}}}+ \|  \Lambda_-\|_{H^{m-1}_{{x}_\mm{v}}} \| \nabla \Lambda_+\|_{H^{m-1}_{{x}_\mm{v}}}\right\|_{L^2_{x_1}}\| \nabla (u^{\mm{d}}, h^{\mm{d}})\|_{{m-1}}. 
\end{align*}  
Consequently, inserting the above estimate into \eqref{25100811saf45}, then integrating the resulting inequality over $(0,\infty)$, and finally using the initial condition \eqref{1asdf.4}$_4$, we arrive at 
\begin{align*} 
&  \sup_{0 \leqslant t<\infty}\{\|\nabla(u^{\mm{d}}, h^{\mm{d}})(t)\|_{m-2}^2 \}
+\varepsilon \|\nabla^2  (u^{\mm{d}}, h^{\mm{d}})\|_{L^2_\infty H^{m-2}}^2  \nonumber 
\\
& \lesssim   \left\|\|  \Lambda_+\|_{H^{m-1}_{{x}_\mm{v}}} \| \nabla \Lambda_-\|_{H^{m-1}_{{x}_\mm{v}}}+ \|  \Lambda_-\|_{H^{m-1}_{{x}_\mm{v}}} \| \nabla \Lambda_+\|_{H^{m-1}_{{x}_\mm{v}}}\right\|_{L^2_{T,x_1}}\| \nabla (u^{\mm{d}}, h^{\mm{d}})\|_{L^2_\infty H^{m-1}}
\end{align*}
 Inserting \eqref{2510071241} into the above estimate, and then using \eqref{2504160945}, \eqref{2511061414}, we have
\begin{align*}
& \sup_{0 \leqslant t<\infty}\{\|\nabla(u^{\mm{d}}, h^{\mm{d}})(t)\|_{m-2}^2\}  
+ \varepsilon\|\nabla^2  (u^{\mm{d}}, h^{\mm{d}})\|_{L^2_\infty H^{m-2}}^2  \nonumber 
\\
&\lesssim  \sup_{0 \leqslant t<\infty}\{\|  (u,h)(t) \|_{m}\}\| \nabla (u,h) \|_{L^2_\infty H^{m}}
\Bigg(\sup_{0 \leqslant t<\infty}\{\|  (u,h)(t) \|_{m}\nonumber \\
&+\left(\sup_{0 \leqslant t<\infty}\{\|  (u,h)(t) \|_{m}\}+
\varepsilon \right)^{1/2}\| \nabla (u^{\mm{d}}, h^{\mm{d}})\|_{L^2_\infty H^{m-1}}\Bigg)\nonumber \\
&\lesssim \varepsilon^{5/2}(\|(v^0,H^0)\|_m^{3}+\|(v^0,H^0)\|_m^{ 7/2}),
\end{align*}
which yields \eqref{1.12}. This completes the proof of Theorem \ref{thma2}.
 
\section{Proof of Theorem \ref{tsafasdhmdsfa1}}   
We will divide the proof of Theorem \ref{tsafasdhmdsfa1} by two steps, which are recorded in the following two sections, resp..

\subsection{An integral formula for linear solutions}
Obviously, the linear problem \eqref{2602102212} can be rewritten as follows
\begin{equation}
\begin{cases}
\mathbf{U}_t^\mm{L}= \mathcal{L} \mathbf{U}^\mm{L}, \\
\mathbf{U}^\mm{L}|_{t=0}=\mathbf{U}^0 ,
\end{cases}
\label{2602102256}
\end{equation}
where we have defined that 
$$\mathbf{U}^\mm{L}=\left(\begin{matrix}
v^\mm{L} \\
H^\mm{L}\\
\end{matrix} \right),\ \mathcal{L} =\left(\begin{matrix}
\mu \Delta & \varepsilon^{-1} \partial_1 \\
\varepsilon^{-1} \partial_1 &  \nu \Delta
\end{matrix}\right)\mbox{ and } \mathbf{U}^0=\left(\begin{matrix}
v^0 \\
H^0\\
\end{matrix} \right),$$

We denote the Fourier transformation of a function $f$ by $\widehat{f}$ and the $n\times n$ unit identity matrix by $\mathbb{I}_n$. Applying Fourier transformation to the linear problem \eqref{2602102256}, thus there holds
\begin{equation}
    \begin{cases}
       \widehat{\mathbf{U}}_t = \mathbf{A} \widehat{\mathbf{U}},\\
        \widehat{\mathbf{U}}|_{t=0}= \widehat{\mathbf{U}^0},
    \end{cases}
    \label{2602111}
    \end{equation}
where we have defined that
\begin{equation*}
\mathbf{A}:=\left(\begin{matrix}
-\mu |\xi|^2 \mathbb{I}_n& \mm{i} \varepsilon^{-1} \xi_1 \mathbb{I}_n \\
\mm{i}\varepsilon^{-1} \xi_1  \mathbb{I}_n&  -\nu |\xi|^2  \mathbb{I}_n
        \end{matrix}\right)\mbox{ with }\mm{i}^2=-1.
\end{equation*}
We have the following conclusion for the initial value problem above.
\begin{lem} \label{lem:pointwise}
We define that $\delta:=((\mu-\nu)|\xi|^2)^2-4(\varepsilon^{-1}\xi_1)^2$ for given $\xi \in \mathbb{R}^n$,
\begin{align}
  \psi(s):=\begin{cases}
 s^{-1}\sinh s&\mbox{for }s\neq 0;\\
 1&\mbox{for } s=0, \end{cases}\ 
\tilde{\Lambda}:=\begin{cases}
    |\sqrt{\delta }|&\mbox{for }\delta\geqslant 0;\\
    \mm{i}|\sqrt{-\delta }|&\mbox{for }\delta<0
    \end{cases}\mbox{ and } \Lambda:= t\tilde{\Lambda}.
 \label{2602121600}
 \end{align}
Then the solution $\widehat{\mathbf{U}}$ of \eqref{2602111} is uniquely given as follows
\begin{equation}
    \widehat{\mathbf{U}}=\mathbf{B}\widehat{\mathbf{U}^0},
    \label{26021113041}
\end{equation}
where we have defined that
\begin{align} \label{2602121619}
\mathbf{B}:=\left(
 \begin{matrix}
  b_{11}\mathbb{I}_n & b_{12} \mathbb{I}_n  \\
  b_{21}  \mathbb{I}_n &  b_{22}  \mathbb{I}_n
 \end{matrix}\right)
 \end{align}
  and   
  \begin{align*}\begin{cases}
   b_{11} := e^{-(\mu+\nu)t|\xi|^2/{2}}( \cosh(\Lambda/2)-(\mu-\nu)t|\xi|^2\psi(\Lambda/2)/2),\\
    b_{12}:= b_{21} =  \mm{i}\varepsilon^{-1} t\xi_1 e^{-(\mu+\nu)t|\xi|^2/{2}}\psi(\Lambda/2) ,\\
    b_{22}:= e^{-(\mu+\nu)t|\xi|^2 /{2}} (\cosh(\Lambda/2)+(\mu-\nu)t|\xi|^2\psi(\Lambda/2)/2).
    \end{cases}
\end{align*}
Moreover each entry $b_{ij}$ in \eqref{2602121619} satisfies
\begin{equation}
\label{2602121346}
|b_{ij}(\xi)| \lesssim e^{-\min\{\mu, \nu\}t |\xi|^2} (t|\xi|^2+1)\mbox{ where }i,\ j \in \{1,2\}.
\end{equation}
\end{lem}
\begin{rem} Here $\sinh s$ and $\cosh s$ represent the hyperbolic sine and the hyperbolic cosine functions rep.,
i.e.
$$\sinh s=\frac{e^s-e^{-s}}{2}\mbox{ and }\cosh s=\frac{e^s+e^{-s}}{2}.$$
\end{rem}
\begin{proof}  Obviously the solution $\widehat{\mathbf{U}}$ of \eqref{2602111} is uniquely given by the formula \eqref{26021113041} by Lemma \ref{2602121343}. Next we deduce the estimates in \eqref{2602121346}.

It is easy to check that
\begin{align}
\label{2602122231}
|\psi(s)|\lesssim e^{|\mm{Re} s|}\mbox{ for any }s\in \mathbb{C}.
\end{align}
In addition, by the definition of $\tilde{\Lambda}$, clearly
\begin{equation*} 
|\mm{Re} \, \tilde{\Lambda}| \leqslant |\mu-\nu||\xi|^2.
\end{equation*} 
Thanks to the above two estimates, we easily get
\begin{align*}
|b_{ii}(\xi)|&\lesssim  e^{-t(\mu+\nu)|\xi|^2/{2}} (|\cosh(\Lambda/2)|+|\psi(\Lambda/2)|)(t|\xi|^2+1 )
\\
&\lesssim e^{-t(\mu+\nu)|\xi|^2/{2} + |\mm{Re}\, \Lambda|/2} (t|\xi|^2+1 )\lesssim  e^{(|\mu-\nu|-(\mu+\nu))t|\xi|^2/2}(t|\xi|^2+1 )\\
&\lesssim  e^{-t\min\{\mu, \nu\}|\xi|^2}(t|\xi|^2+1 ),  
\end{align*}
which immediately yields \eqref{2602121346} for the case $i=j$.

Now we turn to estimating for $b_{ij}$ with $i\neq j$. Firstly, there exists an absolute constant $c\geqslant 1$ such that 
\begin{align}\varepsilon^{-1} |\xi_1|\leqslant c |\mu-\nu| |\xi|^2\leqslant (c \varepsilon^{-1} |\xi_1|)^2
\label{2602122228}
\end{align} 
for any $\varepsilon^{-1}$ and for any  $\xi$ satisfying   
\begin{align*}
\varepsilon^{-1} |\xi_1| >  |\tilde{\Lambda}|.
\end{align*}  
Exploiting \eqref{2602122231} and the first inequality in \eqref{2602122228}, we have 
\begin{equation}
\label{2602121955}
|b_{ij}|\lesssim t|\xi|^2 e^{-t(\mu+\nu)|\xi|^2/{2} + |\mm{Re}\, \Lambda|/2} \lesssim t|\xi|^2  e^{-t\min\{\mu, \nu\} |\xi|^2}
\end{equation}  for any $\xi$ satisfying  $\varepsilon^{-1}|\xi_1|>   |\tilde{\Lambda}|$.
In addition, we easily estimate that
\begin{equation}
\label{02121956}
|b_{ij}| \lesssim e^{-t(\mu+\nu)|\xi|^2/{2}} |\sinh(\Lambda/2)| \lesssim  e^{-t\min\{\mu, \nu\} |\xi|^2}
\end{equation}
for any $\varepsilon^{-1}|\xi_1| \leqslant  |\tilde{\Lambda}|$.
In view of \eqref{2602121955} and \eqref{02121956}, we see that \eqref{2602121346} holds for $i\neq j$. This completes the proof of Lemma \ref{lem:pointwise}.
\end{proof}

We denote the Fourier inverse transformation of a function $f(\xi)$ by $\mathcal{F}^{-1}({f}(\xi))$, and then define
$$\mathbf{K}(x,t):=\mathcal{F}^{-1}(\mathbf{B}(\xi,t)),$$
which is called the fundamental solution of the linear equations in \eqref{2602102256}$_1$.
Applying Fourier inverse transformation to \eqref{26021113041}, we obtain the integral formula for linear solutions of \eqref{2602102256}:
\begin{equation}
\label{2602250006}
    {\mathbf{U}}^{\mm{L}}(x,t) =\int\mathbf{K}(y,t) \mathbf{U}^0(x-y) \mm{d}y.
\end{equation}
In particular,  for the initial data $(v^0, H^0)$ satisfying its support being a subset of the slab $S_R$, 
\begin{equation}\label{2602250007}
\|{\mathbf{U}}^{\mm{L}}(x,t)\|_{L^\infty(S_{l R})}\lesssim \|\mathbf{K}(y,t)\|_{L^\infty(S_{(l+1) R})}\|\mathbf{U}^0\|_{L^1}.
\end{equation} 
To obtain Theorem \ref{tsafasdhmdsfa1}, we shall estimate for the fundamental solution $\mathbf{K}$.

\subsection{Estimate for the fundamental solution}

For any given $\theta\in (0,1)$, there exist two constants $d_1$, ${d_2} \in (0,1)$ such that
\begin{equation}
\label{2602111142}
\theta=d_2-d_1(n-1)\mbox{ and } 2d_1+d_2<1.
\end{equation}
We decompose the expression of $\mathbf{K}$ as follows
\begin{align}
    \mathbf{K}(x,t) &=\sum_{k=1}^{3}\frac{1}{(2\pi)^n} \int_{D_k}  \mathbf{B}(\xi,t)  e^{\mm{i}x\cdot \xi} \mm{d}\xi =: \sum_{k=1}^{3}\mathbf{K}_k(x,t),
    \label{2602271634}
\end{align}
where we have defined that, for some positive constant $\tilde{c}$,
\begin{align}
\begin{cases}
D_1: = \{\xi\in \mathbb{R}^n~|~|\xi| \leqslant \tilde{c}\varepsilon^{-d_1},\ \tilde{c} |\xi_1| \geqslant \varepsilon^{d_2}\},\\
D_2:= \{\xi \in \mathbb{R}^n~|~|\xi| \leqslant \tilde{c}\varepsilon^{-d_1},\ \tilde{c} |\xi_1| \leqslant \varepsilon^{d_2}\}, \\
D_3:= \{\xi \in \mathbb{R}^n~|~ |\xi| \geqslant \tilde{c} \varepsilon^{-d_1}\}.
\end{cases}
\end{align}
Consequently, we have the following estimates of $\mathbf{K}_j$ with $j\in \{1,2,3\}$.
\begin{lem}\label{2602250005}
Let $t_0$ be a positive constant and $l\in \mathbb{Z}^+$. For sufficiently small positive constant $\tilde{c}$ only depending on $\mu$ and $\nu$, it holds that
\begin{align}
\label{260213001}
&\|\mathbf{K}_k(t)\|_{L^\infty(S_{(l+1) R})} \begin{cases}
\lesssim_{\theta,R,t_0,l}  \varepsilon^{\theta}&\mbox{for }k=1;\\
 \lesssim_{\theta, t_0}  \varepsilon^{\theta}&\mbox{for }k=2; \\
 \lesssim_{\theta, t_0}  \varepsilon^{\theta} &\mbox{for }k=3, 
\end{cases}\\
& \|\mathbf{K}_1(t)\|_{L^\infty} \lesssim_{t_0}  1 \label{2602271201}
\end{align}for any $t>t_0$ and for any $\varepsilon\in (0,1)$.
\end{lem}
\begin{proof} 
In what follows, we denote the $n$-dimensional measure of a point set $S\subset \mathbb{R}^n$ by $\mm{meas\,} S $ and $C_i^{\ell} :=\ell!/i!(\ell-i)!$, where $i\in \{0,\ldots,\ell\}$ and $\ell\in \mathbb{Z}^+$. Next we estimate for $\mathbf{K}_1$--$\mathbf{K}_3$ in sequence.  

(1) We begin with to estimate for $\mathbf{K}_1$. To this purpose, we always assume $\xi\in D_1$. Recalling the second relation in \eqref{2602111142} and the definition of $D_1$, it holds that
\begin{align}
\label{2602111304}
|\xi|^2 \leqslant \tilde{c}^2 \varepsilon^{-2d_1} \leqslant \tilde{c}^2 \varepsilon^{d_2-1} \leqslant \tilde{c}^3\varepsilon^{-1}|\xi_1| \mbox{ for any }\varepsilon\in (0,1).
\end{align}
Thus, for sufficiently small $\tilde{c}$ only depending on $\mu$ and $\nu$,
\begin{equation*}
   {4\varepsilon^{-2}\xi_1^2 - (\mu-\nu)^2 |\xi|^4} \geqslant 2\varepsilon^{-2}\xi_1^2 .
\end{equation*}
Hence $ \Lambda$
is a purely imaginary number, and 
 \begin{align}
 \label{26021312345}
   \varepsilon^{-1}t  |\xi_1| \lesssim | \Lambda| \lesssim \varepsilon^{-1} t |\xi_1|.
   \end{align}

Notice that
\begin{equation}
\label{2602131525}
 \xi_1\neq 0\mbox{ and } |\xi_1|\leqslant |\xi| \leqslant \tilde{c}^3 \varepsilon^{-1}.
\end{equation}
Exploiting \eqref{26021312345}, \eqref{2602131525} and the formula
\begin{align*}
 \partial_{\xi_1}^{l} \Lambda=(2\Lambda)^{-1}\left(   \partial_{\xi_1}^{\ell}( { (\mu-\nu)^2 |\xi|^4}-4\varepsilon^{-2}\xi_1^2) - \sum_{j=1}^{{\ell}-1}C_j^{\ell} \partial_{\xi_1}^j\Lambda \partial_{\xi_1}^{\ell-j}\Lambda\right) , 
 \end{align*}
 we can estimate that, for sufficiently small $\tilde{c}$ only depending on $\mu$ and $\nu$,
 \begin{equation}
     \varepsilon^{-1}t   \lesssim  |\partial_{\xi_1}  \Lambda| \lesssim \varepsilon^{-1} t  . 
     \label{26021313061}
\end{equation}
and 
\begin{equation}
    |\partial_{\xi_1}^{\ell}  \Lambda| \lesssim_\ell \varepsilon^{-1} t |\xi_1|^{1-{\ell} }  . 
     \label{2602131306}
\end{equation}
Making use of the last inequality in \eqref{2602111304}, \eqref{2602131306} and the formula
$$ \partial_{\xi_1}^{\ell} (\partial_{\xi_1}\Lambda)^{-1}=-(\partial_{\xi_1}\Lambda)^{-1}   \sum_{j=1}^{{\ell}}C_j^{\ell} \partial_{\xi_1}^{j+1}\Lambda \partial_{\xi_1}^{{\ell} -j}(\partial_{\xi_1}\Lambda)^{-1} ,$$  
we easily conclude that
\begin{equation}\label{eq:doverd}
 | \partial_{\xi_1}^{\ell} (\partial_{\xi_1}  \Lambda)^{-1}| \lesssim_\ell (\varepsilon^{-1} t|\xi_1|^{{\ell} })^{-1} \lesssim_{\ell}    t^{-1} \varepsilon^{ 1 - d_2 \ell}.
\end{equation}

Recalling the definitions of $\mathbf{K}_1$ and each entry $b_{ij}$ of $\mathbf{B}$, it is easy to see that  $\mathbf{K}_1$ can be written as a linear combination of
\begin{align} \label{22602271201} 
I_{m}^{\pm} &= \int_{D_1}  e^{\mm{i}x\cdot \xi-(t(\mu+\nu)|\xi|^2 \pm \Lambda )/2}   \chi_m(\xi) \mm{d}\xi\mbox{ for } m\in \{1, 2,3\},
\end{align}
where we have defined that 
\begin{equation*}
   \chi_1:= 1,\ \chi_2:= \frac{t|\xi|^2}{\Lambda}\mbox{ and } \chi_3:= \frac{\varepsilon^{-1}t\xi_1}{\Lambda}\mbox{ with }t>t_0.
\end{equation*}
Making use of \eqref{26021312345}--\eqref{2602131306} and
 the formulas
$$\begin{cases} 
\partial_{\xi_1}^{\ell} \chi_2=\Lambda^{-1}\left( t\partial_{\xi_1}^{\ell}|\xi|^2- \displaystyle \sum_{j=1}^{{\ell}-1}C_j^{\ell} \partial_{\xi_1}^j\Lambda \partial_{\xi_1}^{{\ell} -j}\chi_2\right),\\
\displaystyle \partial_{\xi_1}^{\ell} \chi_3=\Lambda^{-1}\left( \varepsilon^{-1} t\partial_{\xi_1}^{\ell}\xi_1- \sum_{j=1}^{{\ell}-1}C_j^{\ell} \partial_{\xi_1}^j\Lambda \partial_{\xi_1}^{{\ell} -j}\chi_3\right),
\end{cases}$$
 there holds that
\begin{equation}\label{eq:doverd-b}
      {|\partial_{\xi_1}^{\ell-1} \chi_m|}  \lesssim_{\ell} 
    |\xi_1|^{1-{\ell} }\lesssim_{\ell} 
    \varepsilon^{-d_2{(\ell-1)} }\mbox{ for }m\in \{1,2,3\}.
\end{equation}

 Let $\partial D_1$ be the boundary of $D_1$. Noting that $D_1$ is a Lipschitz domain, thus  we have by the integration by parts with respective to $\xi_1$,
\begin{align*}
    I_m^+ &=  \int_{D_1} \frac{\partial_{\xi_1} e^{\Lambda /2}}{\partial_{\xi_1} (\Lambda/2)} e^{\mm{i}x\cdot \xi-t(\mu+\nu)|\xi|^2/2  } \chi_m(\xi)\mm{d}\xi \nonumber \\
    &= -2\int_{D_1}   e^{\Lambda /2} \partial_{\xi_1} \left(\frac{e^{ \mm{i}x\cdot \xi-t(\mu+\nu)|\xi|^2/2  } \chi_m(\xi) }{\partial_{\xi_1}\Lambda}\right) \mm{d}\xi+2\int_{\partial {D_1} } \frac{ e^{\Lambda /2} e^{ \mm{i}x\cdot \xi-t(\mu+\nu)|\xi|^2/2  } \chi_m(\xi) }{\partial_{\xi_1} \Lambda}\vec{n}_1\mm{d}\xi,
\end{align*}
where $\vec{n}_1$ denotes the first component of outer unit vector $\vec{n}$ on the boundary of $D_1$.
We define the two operators $\mathcal{T}_1$ and $\mathcal{T}_2$ as follows
\begin{equation*} 
    \mathcal{T}_1f: = -2\partial_{\xi_1}\left(\frac{f}{\partial_{\xi_1}\Lambda}\right)\mbox{ and } 
\mathcal{T}_2f: = -\frac{2\partial_{\xi_1}{f}}{{\partial_{\xi_1}\Lambda}}.
\end{equation*}
Repeated integration by parts gives
\begin{align*}
    I_m^+ = &\int_{D_1}  e^{\Lambda /2} \mathcal{T}^{\ell} _1\left({e^{\mm{i}x\cdot \xi-t(\mu+\nu)|\xi|^2/2  } \chi_m(\xi)}\right)\mm{d}\xi \nonumber \\
    &+\sum_{j=1}^{\ell}\int_{\partial D_1}  e^{\Lambda /2} \mathcal{T}^{j-1}_2 \left(\frac{2e^{ \mm{i}x\cdot \xi-t(\mu+\nu)|\xi|^2/2  } \chi_m(\xi) }{\partial_{\xi_1} \Lambda}\right) \vec{n}_1 \mm{d}\xi,
\end{align*}
where $\mathcal{T}_2^0f:=f$.

From now on, we assume that  $x\in S_{(l+1) R}$. Noting that 
\begin{align}
\left|\partial_{\xi_1}^\ell e^{\mm{i}x\cdot \xi-t(\mu+\nu)|\xi|^2/2}\right|\lesssim_{R,t_0,l,\ell} 1, 
\label{2602141742}
\end{align}
and then making use of \eqref{eq:doverd}, \eqref{eq:doverd-b} and \eqref{2602141742},  we have the pointwise upper bound 
\begin{equation*}
    \left|\mathcal{T}^{\ell} _1\left({e^{\mm{i}x\cdot \xi -t(\mu+\nu)|\xi|^2/2 } \chi_m(\xi)}\right) \right| \lesssim_{R,t_0,l,{\ell}}\varepsilon^\ell \sum_{j=0}^\ell \varepsilon^{-d_2j} \lesssim_{R,t_0,l,{\ell}}\varepsilon^{(1-d_2){\ell} },  
\end{equation*} 
where $\varepsilon\in (0,1)$.
Similarly, ones also have 
\begin{align*}
    \sum_{j=1}^{\ell}\left|\mathcal{T}^{j-1}_2  \left(\frac{{e^{\mm{i}x\cdot \xi-t(\mu+\nu)|\xi|^2/2  } \chi_m(\xi) }}{\partial_{\xi_1}\Lambda}\right) \right| \lesssim_{R,t_0,l,{\ell}} &\sum_{j=1}^{\ell} \varepsilon^{(1-d_2)(j-1) +1} 
   \lesssim_{R,t_0,l,{\ell}} \varepsilon,
\end{align*}
where $\varepsilon\in (0,1)$.

As a consequence, we get 
\begin{equation*}
    |I^+_m| \lesssim_{R,t_0,l,{\ell}} \varepsilon^{(1-d_2){\ell} } \text{meas\,} D_1+ 
  \varepsilon  \mm{meas}_{n-1} \partial D_1  \lesssim_{R,t_0,l,{\ell}} \varepsilon^{(1-d_2){l} -d_1 n}+\varepsilon^{1 -d_1 (n-1)},
\end{equation*}
where $\mm{meas}_{n-1} D_1$ represents the $(n-1)$-dimensional measure of $D_1$.
Obviously the upper bound of $|I^-_m|$ is also the one of $|I^+_m|$. Therefore we immediately have,
for any given ${\ell} \in \mathbb{Z}^+$,
\begin{equation*} 
    |\mathbf{K}_1| \lesssim_{R,t_0,l,{\ell}} \varepsilon^{(1-d_2){\ell}  -d_1 n} +\varepsilon^{1 -d_1 (n-1)}\mbox{ for any }\varepsilon\in (0,1),
\end{equation*}
which, for some sufficiently large ${\ell} $, yields \eqref{260213001} with $k=1$.

(2) Now we turn to estimating for $\mathbf{K}_2$. Recalling the definition of $D_2$, we have
\begin{equation*}
 |D_2| \lesssim \varepsilon^{d_2-d_1(n-1)}\lesssim \varepsilon^{\theta}.
\end{equation*}
Recalling \eqref{2602121346}, and then using  the fact
\begin{align}
\label{2602271157}
e^{- t \min\{\mu, \nu\}|\xi|^2/2} t|\xi|^2\lesssim 1,
\end{align}
we have
\begin{equation*} \label{mij-R2}
    |b_{ij}(\xi,t)| \lesssim 1.
\end{equation*}
Hence it obviously holds that
\begin{equation*} 
    |\mathbf{K}_2| \lesssim \text{meas\,} D_2 \lesssim \varepsilon^\theta ,
\end{equation*}
which yields \eqref{260213001} with $k=2$.

(3) Now we estimate for $\mathbf{K}_3$. It is easy to check that  
\begin{align*}  e^{- \min\{\mu, \nu\}\tilde{c}^2 t_0\varepsilon^{-2d_1}/2}\lesssim_{\theta,t_0}\varepsilon^{\theta} \mbox{ for any }\varepsilon\in (0,1)
\end{align*}  Exploiting \eqref{2602271157} and the above estimate, we obtain the upper bound
\begin{align*}
    |\mathbf{K}_3| \lesssim &\int_{|\xi| \geqslant \tilde{c} \varepsilon^{-d_1}} e^{- t \min\{\mu, \nu\}|\xi|^2} (t|\xi|^2+1) \mm{d}\xi  \\
    \lesssim &e^{- \min\{\mu, \nu\}\tilde{c}^2t_0 \varepsilon^{-2d_1}/2}\lesssim_{\theta,t_0}\varepsilon^{\theta},
\end{align*}
which yields \eqref{260213001} with $k=3$. 

(4)  Finally, recalling \eqref{22602271201} and \eqref{eq:doverd-b}, we have
$$\|\mathbf{K}(t)\|_{L^\infty}\lesssim \int_{D_1}  e^{ -t(\mu+\nu)|\xi|^2/2}  \mm{d}\xi\lesssim \int_{\mathbb{R}^n}  e^{ -t(\mu+\nu)|\xi|^2/2}  \mm{d}\xi\lesssim_{t_0} 1\mbox{ for any }t>t_0, $$
which yields \eqref{2602271201}.
This completes the proof of Lemma \ref{2602250005}.
\end{proof} 

Now making use of Lebesgue dominated convergence theorem, \eqref{2602250006}, \eqref{2602250007}, \eqref{2602271634} and Lemma \ref{2602250005}, we immediately obtain Theorem \ref{tsafasdhmdsfa1}. 
 
  \appendix
  \section{Formula of solutions for the 1-order linear differential system}
 \renewcommand\thesection{A}
  
  This appendix is devoted to the derivation for the formula of solutions of the 1-order linear differential system.
  \begin{lem}\label{2602121343}
  Let $\ell\in\mathbb{Z}^+$, $\mathbb{I}_\ell$ be a $\ell\times \ell$ unit identity matrix, $a_{ij}\in \mathbb{C} $, $\delta={(a_{11}-a_{22})^2 + 4a_{12}a_{21}}\in \mathbb{R}$, $\digamma_i \in \mathbb{C}^\ell$,
  $$\mathbf{A}
  =\left(
 \begin{matrix}
  a_{11}\mathbb{I}_\ell & a_{12} \mathbb{I}_\ell  \\
  a_{21}  \mathbb{I}_\ell &  a_{22}  \mathbb{I}_\ell
 \end{matrix}\right) 
\mbox{ and } \lambda:=
\begin{cases}
\sqrt{\delta}&\mbox{for }\delta\geqslant 0;\\
\mm{i}\sqrt{-\delta}&\mbox{for }\delta<0. \end{cases}
$$  then the initial value problem
 \begin{align}
 \begin{cases}
 f'=\mathbf{A} f\mbox{ for }t>0;\\
 f|_{t=0}=f^0:=(\digamma_1^\top,\digamma_2^\top)^\top&
 \label{2602121052}
 \end{cases}
 \end{align}
 has the unique solution, which is given by the following formula
\begin{align}
\label{2602121122}
f=e^{(a_{11}+a_{22})t/{2}}\mathbf{B}f^0 ,
\end{align}
where we have defined that 
\begin{align*}&\mathbf{B}=\left(
 \begin{matrix}
 b_{11}\mathbb{I}_\ell & b_{12} \mathbb{I}_\ell  \\
  b_{21}  \mathbb{I}_\ell &  b_{22}  \mathbb{I}_\ell
 \end{matrix}\right) ,\\
 &b_{11} : =    \cosh\left(\frac{\lambda t}{2}\right)
 + \frac{({a_{11}-a_{22}})t}{2}\psi\left(\frac{\lambda t}{2}\right) 
 ,\ 
 b_{12}:
 = a_{12}t  \psi\left(\frac{\lambda t}{2}\right) , \\
& b_{21}:
  =  a_{21}t \psi\left(\frac{\lambda t}{2}\right), \ 
 b_{22}:
  =   
 \cosh\left(\frac{\lambda t}{2}\right)
 + \frac{(a_{22}-a_{11})t}{2}\psi\left(\frac{\lambda t}{2}\right) ,\\
 &\mbox{and } \psi \mbox{ is defined by }\eqref{2602121600}.
\end{align*} 
  \end{lem}
\begin{proof}
Let 
\begin{align}
\label{2602271236}
\mathbf{C} = \mathbf{A} - \frac{a_{11}+a_{22}}{2} \mathbb{I}_\ell
 =\left(
 \begin{matrix}
  ({a_{11}-a_{22}}) \mathbb{I}_\ell/{2} & a_{12}  \mathbb{I}_\ell\\
  a_{21}  \mathbb{I}_\ell & ({a_{22}-a_{11}})/{2} \mathbb{I}_\ell
 \end{matrix}\right).
\end{align} 
We can compute out that
\begin{align}
\mathbf{C}^{2} =\delta\mathbb{I}_{2\ell}/4.
\label{2602111605}
\end{align}
Here and in what follows, $\mathbb{I}_{2\ell}$ represents a $2\ell\times 2\ell$  unit identity matrix.
 Exploiting \eqref{2602111605}, we have 
\begin{align}
 e^{\mathbf{C}}& = \sum_{k =0}^\infty \frac{\mathbf{C}^k}{k!}=\sum_{k =0}^{\infty}\frac{1}{(2k)!}\left(\frac{\delta}{4}\right)^{k}\mathbb{I}_{2\ell}+\sum_{k =0}^{\infty}\frac{1}{(2k+1)!}\left(\frac{\delta}{4}\right)^{k}\mathbf{C}\nonumber\\
 & =\cosh \left(\frac{\lambda}{2}\right)\mathbb{I}_{2\ell }
 + \psi \left(\frac{\lambda}{2}\right) \mathbf{C}=\mathbf{B}|_{t=1}.
 \label{2602121123}
\end{align}
Utilizing \eqref{2602271236} and \eqref{2602121123}, we obtain
\begin{align}
\label{2602121053}
e^{\mathbf{A}}  =e^{({a_{11}+a_{22}})/{2}} e^{\mathbf{C}} = 
e^{(a_{11}+a_{22})/{2}}\mathbf{B}|_{t=1} 
\end{align}
Thus replacing $a_{ij}$ by $a_{ij}t$ in \eqref{2602121053} yields 
\begin{align}
\label{26021210d53}
e^{\mathbf{A}t}  = 
e^{(a_{11}+a_{22})t/{2}}\mathbf{B} .
\end{align} 
We mention that we have used the properties of the power series of matrices, which can be found in Section 18 in Chpater IV in \cite{WWODE148}, for the derivation of \eqref{2602121123} and \eqref{2602121053}. 

It is well-know that 
$f=e^{\mathbf{A} t}f^0$ is the unique solution of \eqref{2602121052} (see Part VI in Section 18 in \cite{WWODE148}). Thanks to the formula \eqref{26021210d53},  $f$ given by the formula \eqref{2602121122} is the unique solution of \eqref{2602121052}. This proof is complete.  
\end{proof}
\vspace{4mm}
\noindent\textbf{Acknowledgements.} The research of Fei Jiang was partially supported by NSFC (Grant Nos. 12525109, 12371233 and 12231016) and the NSF of Fujian Province of China (Grant  Nos. 2024J011011 and 2022J01105), the research of Xiao Ren by NSFC (No. 62588101) and the National Key R$\&$D Program of China (No. 2023YFA1010700), and the research of Yi Zhou by NSFC (No. 12171097).
	
\renewcommand\refname{References}
\renewenvironment{thebibliography}[1]{%
\section*{\refname}
\list{{\arabic{enumi}}}{\def\makelabel##1{\hss{##1}}\topsep=0mm
\parsep=0mm
\partopsep=0mm\itemsep=0mm
\labelsep=1ex\itemindent=0mm
\settowidth\labelwidth{\small[#1]}
\leftmargin\labelwidth \advance\leftmargin\labelsep
\advance\leftmargin -\itemindent
\usecounter{enumi}}\small
\def\newblock{\ }
\sloppy\clubpenalty4000\widowpenalty4000
\sfcode`\.=1000\relax}{\endlist}
\bibliographystyle{model1b-num-names}

\end{CJK*}
\end{document}